\documentclass[10pt]{article}

\DeclareMathAlphabet{\mathpzc}{OT1}{pzc}{m}{it}

\usepackage{amsmath}
\usepackage{amssymb}
\usepackage{amsfonts}
\usepackage{amsthm}
\usepackage{mathrsfs}
\usepackage{enumerate}
\usepackage[pdftex]{color, graphicx}

\newcommand{\cC}{{\mathcal C}}

\newcommand{\pa}{\parallel}
\newcommand{\cd}{\cdot}

\newcommand{\li}{l_\infty}

\newcommand{\C}{\mathcal{C}}

\newcommand{\Cp}{\mathcal{C}_P}
\newcommand{\N}{\mathcal{N}}
\newcommand{\Nh}{\mathcal{N}_H}
\newcommand{\Npp}{\mathcal{N}_{P'}}
\newcommand{\I}{\mathcal{I}}

\allowdisplaybreaks

\newtheorem{thm}{Theorem}[section]
\newtheorem{lem}[thm]{Lemma}
\newtheorem{defn}[thm]{Definition}
\newtheorem{prop}[thm]{Proposition}
\newtheorem{cor}[thm]{Corollary}

\begin{document}

\renewcommand{\thefootnote}{\arabic{footnote}}
 	
\title{Synthetic foundations of cevian geometry, IV:\\ The TCC-Perspector Theorem}

\author{\renewcommand{\thefootnote}{\arabic{footnote}}
Igor Minevich and Patrick Morton}
\maketitle

\begin{section}{Introduction.}

In the first three parts of this series of papers (\cite{mm1}, \cite{mm2}, \cite{mm3}), we have studied generalizations of some classical configurations for an ordinary triangle $ABC$ and a variable point $P$ in the extended Euclidean plane, using a mixture of affine and projective methods.  As in previous papers, we assume $P$ does not lie on the sides of $ABC$ or its anticomplementary triangle $K^{-1}(ABC)$, where $K$ is the complement map for $ABC$.  Letting $\iota$ denote the isotomic map for $ABC$, we defined $P'=\iota(P)$ and $Q=K(P')=K \circ \iota(P)$, the latter point being the isotomcomplement of $P$ and a generalization of the incenter.  We also studied what we called the generalized orthocenter $H$ and generalized circumcenter $O$ associated to $ABC$ and $P$.  The point $H$ is defined as the intersection of the lines through $A,B,C$ which are parallel, respectively, to $QD,QE,QF$, where $DEF$ is the cevian triangle of $P$ for $ABC$ (the diagonal triangle of the quadrangle $ABCP$, with $D=AP \cdot BC, E=BP \cdot CA, F = CP \cdot AB$).  The point $O$ is defined similarly, as the intersection of the lines through the midpoints $D_0,E_0,F_0$ of the sides $BC, CA, AB$, which are parallel to $QD, QE,QF$.  It is easy to see that $O=K(H)$.  We proved various relationships between the circumconic $\tilde \cC_O$ of $ABC$ whose center is $O$, the nine-point conics $\Nh$ and $\N_{P'}$ of the quadrangles $ABCH$ and $ABCP'$ with respect to the line at infinity, and the inconic $\mathcal{I}$ which is tangent to the sides of $ABC$ at the points $D,E,F$.  In particular, if $T_P$ denotes the affine map taking $ABC$ to $DEF$, and $T_{P'}$ denotes the affine map taking $ABC$ to $D_3E_3F_3$, the cevian triangle for $P'$, many of these relationships can be expressed in terms of $T_P$ and $T_{P'}$, in combination with the complement map.  For example, the circumconic $\tilde \cC_O=T_{P'}^{-1}(\Npp)$ is the nine-point conic for the quadrangle $Q_aQ_bQ_cQ$, where $Q_aQ_bQ_c=T_{P'}^{-1}(ABC)$ is the anticevian triangle of the point $Q$ with respect to $ABC$, and its center is given by $O=T_{P'}^{-1} \circ K(Q)$.  To take another example, we also showed that the affine map $\textsf{M}=T_P \circ K^{-1} \circ T_{P'}$ is a homothety or translation taking the circumconic $\tilde \cC_O$ to the inconic $\I$.  Its center $S$ is a generalization of the classical insimilicenter. \medskip

In this part of our series on cevian geometry we define the generalized isogonal map $\gamma_P$ for an ordinary triangle $ABC$ and a point $P$ not on the sides of $ABC$ or $K^{-1}(ABC)$.  All the relationships holding between the classical orthocenter and circumcenter also hold for the generalized notions, but with $\gamma_P$ replacing the usual isogonal map $\gamma$ for $ABC$.  Thus, $\gamma_P(O)=H$ and $\gamma_P(l_\infty)=\tilde \cC_O$, i.e. the $P$-isogonal image of the line at infinity is the circumconic $\tilde \cC_O$.  These relationships reduce to the classicial ones when $P=Ge$ is the Gergonne point of $ABC$.  \medskip

We also define the corresponding notions of pedal triangle and pedal conic (with respect to $P$) for a point $R_1$ and its image $\gamma_P(R_1)=R_2$, and show that $\gamma_P$ is a reciprocal conjugation, in the sense of \cite{dvl}.  It turns out that the pole of this reciprocal conjugation, defined by Dean and van Lamoen \cite{dvl}, is the isotomcomplement of our generalized orthocenter $H$.  In addition, our development ties together and gives a synthetic context for the papers \cite{dau}, \cite{dvl}, \cite{fg}, and \cite{gu}.  \medskip

The main result of this paper is the TCC-Perspector Theorem in Section 4, which says that the (classical) isogonal conjugate $\gamma(H)$ is the perspector of the tangential triangle (T) of $ABC$ and the circumcevian triangle (CC) of the isogonal conjugate $\gamma(Q)$ of $Q$, both taken with respect to the circumcircle of $ABC$.  This relates the points $H$ and $Q$, which are the generalizations of the orthocenter and incenter, respectively, to the tangential and circumcevian triangles in the classical sense.  All but a few of our results are aimed at proving this theorem.  \medskip

Along the way we also prove various other results, such as the relation $\gamma_P(P)=S$ for the center of the map $\textsf{M}$ mentioned above (see \cite{mm3}, Theorem 3.4), and the relations $\gamma_P = \delta_P \circ \iota \circ \delta_P=\delta_H$ between $\gamma_P$ and the map $\delta_P$ defined in Part I (\cite{mm1}, Corollary 3.6).  We also prove a natural generalization of Simson's Theorem (Theorem \ref{thm:Sim}) in this context.  The proof of this theorem, as well as the proofs for the properties of the pedal triangles and pedal conics, are heavily projective, and give alternate approaches to the classical theorems that they generalize.  Finally, we discuss the point $\gamma_P(G)$, which is a generalization of the symmedian point, in Theorem \ref{thm:persG}.  \medskip

As in our previous papers, we refer to \cite{mm1}, \cite{mm2}, and \cite{mm3} as I, II, and III, respectively.  The generalized orthocenter and circumcenter were first defined in \cite{mmq}, and that paper also contains proofs of the affine formulas for $H$ and $O$.  All of our proofs are synthetic, except for the proof of Proposition \ref{prop:Q^2}, in which we make very limited use of barycentric coordinates to relate $\gamma_P$ to the discussion in \cite{dvl}.  We will give a synthetic proof of this connection elsewhere.  We do, however, give a synthetic proof of Corollary \ref{cor:Q2}, using the relationship proved in Theorem \ref{thm:gammadelta}. We do not make use of the results of \ref{prop:Q^2} and \ref{cor:Q2} anywhere else in the paper.

\end{section}

\begin{section}{The generalized isogonal map.}

Given a point $P$ with cevian triangle $DEF$ inscribed in triangle $ABC$, we consider the affine reflections
\[h_a \ \text{about the line $AQ$ in the direction} \ a = EF,\]
\[h_b \ \text{about the line $BQ$ in the direction} \ b = DF,\]
\[h_c\  \text{about the line $CQ$ in the direction} \ c = DE.\]
Thus, $h_a(X) = X'$, where $XX'$ is parallel to $EF$ and $XX' \cd AQ$ is the midpoint of $XX'$. In this situation the lines $AX$ and $AX'$ are harmonic conjugates with respect to $AQ$ and the line $l_a$ through $A$ which is parallel to $EF$.  From I, Theorem 3.9 and I, Corollary 3.11(a), the line $l_a=T_{P'}^{-1}(BC)$.

\begin{defn}
We define the {\bf generalized isogonal map} $\gamma_P$ as follows: for a given $X \ne A, B, C$, let $\gamma_P(X)$ be the intersection of the lines $Ah_a(X)$, $Bh_b(X)$, and $Ch_c(X)$.  (See Figure \ref{fig:2.1}.)
\end{defn}
If $G$ is the centroid, then $\gamma_G=\iota$ is the isotomic map, while if $P=Ge$ is the Gergonne point, $Q$ is the incenter, and $\gamma_P$ is the isogonal map.  This definition depends on the following theorem.

\begin{thm}
\label{thm:gammaP}
For any $R \ne A, B, C$, the lines $Ah_a(R), Bh_b(R), Ch_c(R)$ are concurrent.
\end{thm}

\begin{proof}
We first state Ceva's theorem in terms of cross-ratios. Let $AR \cd BC = R_1, BR \cd AC = R_2, CR \cd AB = R_3$. We have the cross-ratios
\[A(BC, RQ) = (BC, R_1D_2) = \frac{BR_1}{R_1C} \div \frac{BD_2}{D_2C},\]
\[B(CA, RQ) = (CA, R_2E_2) = \frac{CR_2}{R_2A} \div \frac{CE_2}{E_2A},\]
\[C(AB, RQ) = (AB, R_3F_2) = \frac{AR_3}{R_3B} \div \frac{AF_2}{F_2B}.\]
The condition of Ceva's theorem is therefore equivalent to
\[A(BC, RQ) \cd B(CA, RQ) \cd C(AB, RQ) = 1.\]
On the other hand, if $R_a = h_a(R), R_b = h_b(R), R_c = h_c(R)$, and $S_1 = AR \cd EF, S_a = AR_a \cd EF$, then
\begin{align*}
A(BC, R_aQ) &= (FE, S_aA_0) = \frac{FS_a}{S_aE} = \frac{h_a(FS_a)}{h_a(S_aE)}\\
&= \frac{S_1E}{FS_1} = \frac{1}{(FE, S_1A_0)} = \frac{1}{A(BC, RQ)},
\end{align*}
and similarly
\[B(CA, R_bQ) = \frac{1}{B(CA, RQ)} \quad \text{and} \quad C(AB, R_cQ) = \frac{1}{C(AB, RQ)}.\]
It follows that $1 = A(BC, R_aQ) \cd B(CA, R_bQ) \cd C(AB, R_cQ)$, which implies that the lines $AR_a, BR_b, CR_c$ are concurrent.
\end{proof}

\begin{figure}
\[\includegraphics[width=4.5in]{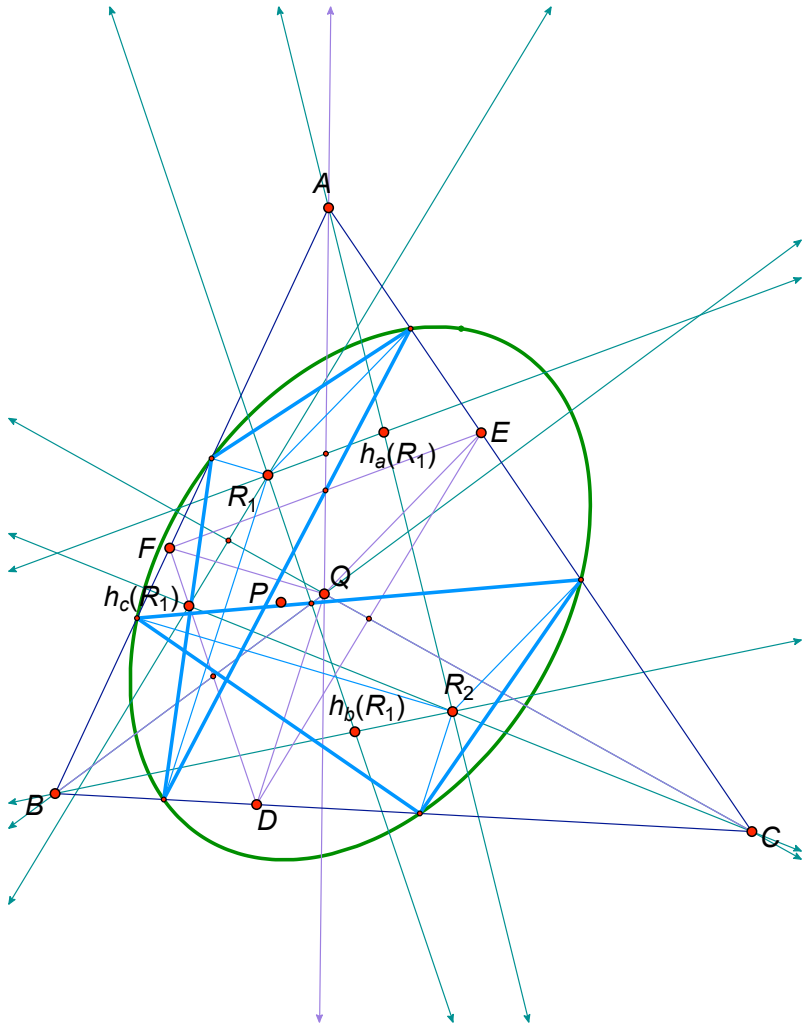}\]
\caption{Pedal triangles and pedal conic for $R_1$ and $R_2=\gamma_P(R_1)$}
\label{fig:2.1}
\end{figure}
\bigskip

It is clear that the map $\gamma_P$ fixes the point $Q$. It is also not difficult to verify that $\gamma_P$ fixes the vertices of the anticevian triangle $Q_aQ_bQ_c=T_{P'}^{-1}(ABC)$ of $Q$ with respect to $ABC$. This is because the sides of the anticevian triangle of $Q$ are parallel to the sides of triangle $DEF$ (see I, Theorem 3.9), so the side through $B$, for instance, is fixed by the map $h_b$. Since the side through $B$ intersects $CQ$ at the vertex $Q_c$, and $CQ$ is fixed by $h_c$, it is clear that $\gamma_P$ fixes $Q_c=Q_cQ_a \cdot CQ$. The same argument applies to the other vertices. It is easy to see that these four points are the only fixed points of $\gamma_P$.  \medskip

As an application we first show the following.

\begin{prop}
\label{prop:PS}
If $S$ is the center of the map $\textsf{M}$ (III, Theorem 3.4), then $\gamma_P(P)=S$.
\end{prop}

\noindent {\bf Remark}. This verifies the assertion made after III, Theorem 3.4, that $S$ coincides with the point $\gamma_P(P)$.

\begin{proof}
We note first that $\textsf{M}^{-1}(D)=T_{P'}^{-1} \circ K \circ T_P^{-1}(D)=T_{P'}^{-1} \circ K(A)=T_{P'}^{-1}(D_0)$, which is the midpoint of the side $Q_bQ_c$ of the anticevian triangle of $Q$ containing the point $A$ (I, Corollary 3.11 and III, Section 2).  If $\textsf{M}^{-1}(D) \neq A$, this implies that $A\textsf{M}^{-1}(D) = Q_bQ_c$ is parallel to $EF$ (I, Theorem 3.9).  Therefore, $\textsf{M}(A\textsf{M}^{-1}(D)) = \textsf{M}(A)D$ is parallel to $EF$ as well, since $\textsf{M}$ is a homothety or translation (III, Theorem 3.4).  Thus $\textsf{M}(A)D$ lies on $EF \cdot l_\infty$.  However, both $D$ and $\textsf{M}(A)$ lie on the inconic $\mathcal{I}$ (III, Theorem 3.4).   Hence, the midpoint of $\textsf{M}(A)D$ lies on $AQ$, which is the conjugate diameter for the direction $EF$ (by I, Theorem 2.4).  This gives that $h_a(D) = \textsf{M}(A)$, so $h_a(AP) = h_a(AD) = A\textsf{M}(A) = AS$.  If $\textsf{M}(B) \neq E$ and $\textsf{M}(C) \neq F$ we also get $h_b(BP)=BS$ and $h_c(CP)=CS$.  Hence, $\gamma_P(P)=S$.  (Note that $S$ is never a vertex for the points $P$ that we are considering.  If $S=A$, for example, then $A = \textsf{M}(A)$ lies on the inconic, which is impossible, since the inconic can't be tangent to a side of $ABC$ at a vertex.)  \smallskip

If $\textsf{M}(A)=D$, but $\textsf{M}(B) \neq E$ and $\textsf{M}(C) \neq F$, then we still have $\gamma_P(P) = h_b(BP) \cdot h_c(CP) = BS \cdot CS = S$.  On the other hand, if $D = \textsf{M}(A)$ then $\textsf{M}(A)$, being symmetrically defined with respect to $P$ and $P'$ ($\textsf{M}=T_P \circ K^{-1} \circ T_{P'} = T_{P'} \circ K^{-1} \circ T_P$ by III, Proposition 3.12 and Lemma \ref{lem:commute} in the Appendix below), lies on both inconics $\mathcal{I}=\mathcal{I}_P=\textsf{M}(\tilde \cC_O)$ and $\mathcal{I}'=\mathcal{I}_{P'}=\textsf{M}(\tilde \cC_{O'})$. Since it also lies on $BC$, which is tangent to both inconics, $\textsf{M}(A)$ equals both $D=D_1$ and $D_3$, so $P$ and $P'$ lie on the median $AD_0$.  If, in addition, $E = \textsf{M}(B)$, then in the same way $E = E_3$, implying that $P = P' = G$, in which case $Q = G$ and $\gamma_P(P) = \gamma_P(G) = G = S$.
\end{proof}

We also note the following relationship between the map $\gamma_P$ and the map $\delta_P$ from I, Corollary 3.6.  Let $\iota$ represent the isotomic map for $ABC$.  Recall from I, Theorem 3.8 that the point $X$ is the perspector of triangles $ABC$ and $A_3B_3C_3$, and is the center of the map $\mathcal{S}=T_P \circ T_{P'}$.

\begin{prop}
\label{prop:gpdp}
\begin{enumerate}[(a)]
\item For any point $P$ not on the sides of $ABC$ or its anticomplementary triangle, $\gamma_P = \delta_P \circ \iota \circ \delta_P$.
\item The center $X$ of the map $\mathcal{S}=T_P \circ T_{P'}$ satisfies $\gamma_P(X)=Q'$.
\item The point $S$ satisfies $\delta_P(S)=\iota(Q')$.
\end{enumerate}
\end{prop}

\begin{proof}
Let $P_1$ and $P_2$ be points distinct from the vertices of $ABC$, let $D_1, D_2$ be their traces on $BC$, and let $D_1', D_2'$ be the traces of $\delta_P(P_1)$ and  $\delta_P(P_2)$ on $BC$.  Then $h_a(AP_1) = AP_2$ if and only if $A_0$ is the midpoint of $T_P(D'_1D'_2)$, and this holds if and only if $D_0$ is the midpoint of $D_1'D_2'$.  The corresponding relations will hold for all the sides of $ABC$ if and only if $\gamma_P(P_1) = P_2$, and also if and only if $\delta_P(P_2)$ is the isotomic conjugate of $\delta_P(P_1)$.  Hence, using the fact that $\delta_P$ is an involution, $\gamma_P(P_1)=P_2=\delta_P \circ \iota \circ \delta_P(P_1)$.  This proves (a).  Part (b) is immediate from this and I, Corollary 3.6, since $\gamma_P(X)=\delta_P \circ \iota \circ \delta_P(X)=\delta_P \circ \iota(P')=\delta_P(P) =Q'$.  For part (c), we have $\delta_P(S) = \iota \circ \delta_P \circ \gamma_P(S)=\iota \circ \delta_P(P)=\iota(Q')$ from Proposition \ref{prop:PS}.
\end{proof}
\smallskip

Proposition \ref{prop:gpdp}(b) shows that $X$ is the generalized isogonal conjugate of the generalized Mittenpunkt: if $P$ is the Gergonne point, then $Q' =X(9)$ is the Mittenpunkt, and $X = X(57)$ in Kimberling's {\it Encyclopedia} \cite{ki2}. This also provides a new proof of the fact that the $X(7)$-ceva conjugate of the incenter (see I, Theorem 3.10) is the isogonal conjugate of the Mittenpunkt.

\begin{defn}
Assume that the point $P$ does not lie on $\iota(\li)$, so that $Q$ is an ordinary point. The vertices of the {\bf pedal triangle for a point $R$ with respect to $P$} are the intersections with the sides $BC, AC, AB$ of lines through $R$ parallel, respectively, to $QD, QE, QF$.  (See Figures \ref{fig:2.1} and \ref{fig:2.3}.)
\end{defn}
\smallskip

\begin{thm}
\label{thm:PedalTriangles}
If the points $R_1$ and $R_2 = \gamma_P(R_1)$ are ordinary points, different from any of the points in the set $\{Q,Q_a,Q_b,Q_c\}$, then their pedal triangles with respect to $P$ are inscribed in a common conic, called the {\bf pedal conic} for $R_1$ and $R_2$ with respect to $P$.
\end{thm}

\begin{proof}
(See Figures \ref{fig:2.1} and \ref{fig:2.2}.)  Let $D_1, D_2$ be the points on $BC$ such that $R_1D_1$ and $R_2D_2$ are parallel to $QD$, and $F_1, F_2$ be the points on $AB$ such that $R_1F_1$ and $R_2F_2$ are parallel to $QF$. Furthermore, let
\[G = AB \cd R_1D_1, H = BC \cd R_1F_1, J = AB \cd R_2D_2, K = BC \cd R_2F_2;\]
and set $S = h_b(R_2)$ (on $BR_1$), $D' = h_b(D_2)$ (on $AB$), $J' = h_b(J) = BC \cd SD'$. \smallskip

\begin{figure}
\[\includegraphics[width=5in]{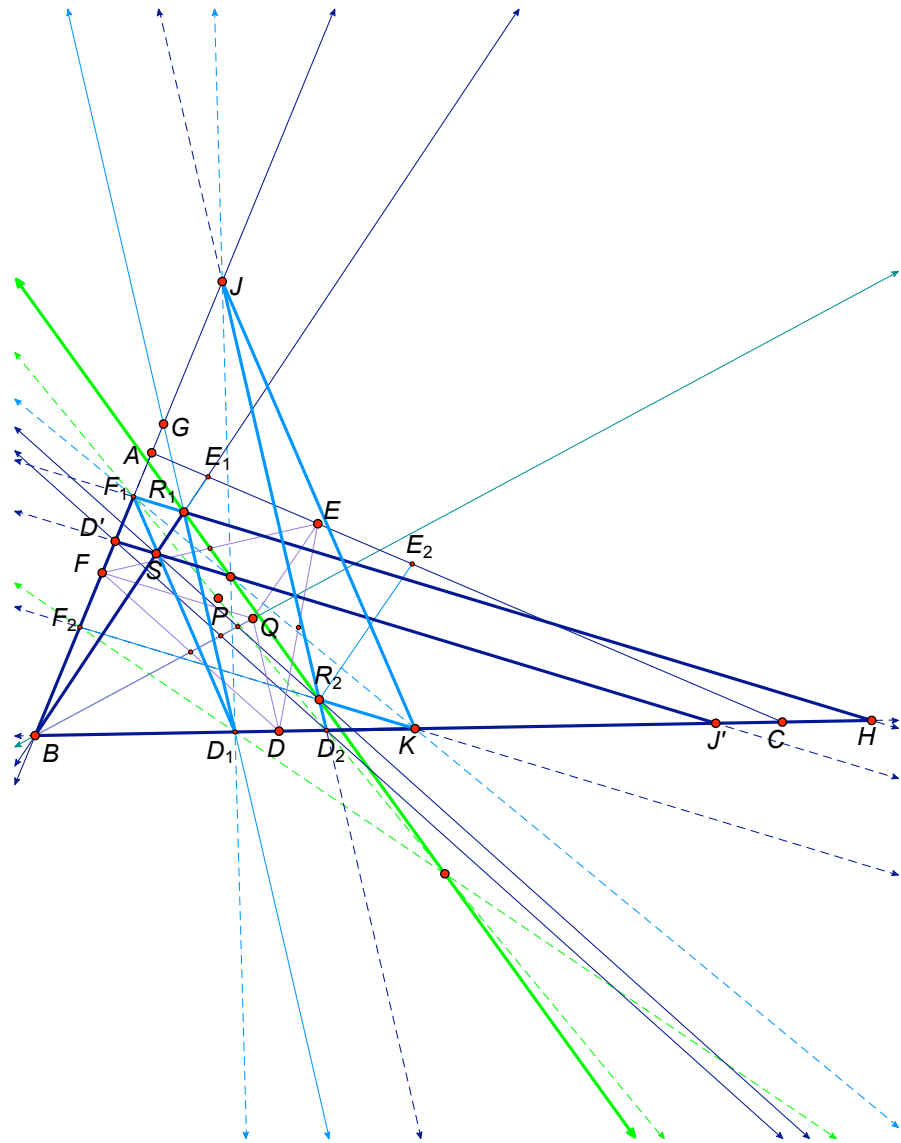}\]
\caption{Proof of Theorem \ref{thm:PedalTriangles}}
\label{fig:2.2}
\end{figure}

{\it Case 1.} The points $G, H, J, K$ are ordinary. Then $SD' \pa h_b(QD) = QF$, since $BQ \cdot DF=B_0$ is the midpoint of $DF$ (I, Theorem 2.4). It follows that $BR_1F_1 \sim BSD'$ and $BR_1H \sim BSJ'$, and therefore
\begin{equation}
\label{eqn:5.1}
\frac{R_1F_1}{R_1H} = \frac{R_1F_1/BR_1}{R_1H/BR_1} = \frac{SD'/BS}{SJ'/BS} = \frac{SD'}{SJ'} = \frac{h_b(R_2D_2)}{h_b(R_2J)} = \frac{R_2D_2}{R_2J}.
\end{equation}
We also have that $R_1D_1H \sim R_2D_2K$ so that $(R_1D_1) \cd (R_2K) = (R_2D_2) \cd (R_1H)$. Putting this together with (\ref{eqn:5.1}) shows that $(R_1D_1) \cd (R_2K) = (R_1F_1) \cd (R_2J)$, and thus
\[\frac{R_1D_1}{R_1F_1} = \frac{R_2J}{R_2K}.\]
Since $R_1D_1 \pa R_2J$ and $R_1F_1 \pa R_2K$, we have that $R_1D_1F_1 \sim R_2JK$ by the $SAS$ criterion for similarity. This implies that $D_1F_1 \pa JK$ and therefore that triangles $R_1D_1F_1$ and $R_2JK$ are perspective from the point $D_1J \cd F_1K$ on the line $R_1R_2$, by the converse of Desargues' theorem. From this we see that the axis of the projectivity taking $F_1JF_2$ to $D_1KD_2$ is the line $R_1R_2$, since $JD_2 \cd F_2K = R_2$. Note that $R_2 \neq D_1J \cd F_1K$, since otherwise $D_1=D_2, F_1=F_2$, hence $R_1=R_2$, contrary to the hypothesis that $R_1$ is not a fixed point of $\gamma_P$.  Therefore, the cross-join $F_1D_2 \cd F_2D_1$ lies on the line $R_1R_2$.\smallskip

{\it Case 2.} One of the points $G, H, J, K$ is infinite. If $G$ is infinite, for example, then $QD \pa AB$, and using the map $h_b$ it follows that $QF \pa BC$, so that all four points are infinite. In this case we have $AB \cd R_1D_1 = J = AB\cd R_2D_2, BC \cd R_1F_1 = K = BC\cd R_2F_2$. Considering the projectivity taking $F_1JF_2$ to $D_1KD_2$, it is easy to see that the axis is the line $R_1R_2$, since $D_1J \cd F_1K = R_1$ and $JD_2\cd F_2K = R_2$. Hence, the cross-join $F_1D_2 \cd F_2D_1$ lies on $R_1R_2$. \smallskip

By similar reasoning applied to the other vertices, we have that $D_1E_2 \cd D_2E_1$ and $E_1F_2 \cd E_2F_1$ also lie on the line $R_1R_2$, where $R_1E_1$ and $R_2E_2$ are parallel to $QE$ with $E_1$ and $E_2$ on $AC$. But then the converse of Pascal's theorem shows that the vertices of the hexagon $D_1E_2F_1D_2E_1F_2$ lie on a conic. This proves the theorem.
\end{proof}

\begin{cor}
The projectivity $\pi: F_1JF_2 \mapsto D_1KD_2$ takes the point $G = AB \cd R_1D_1$ to the point $H = BC\cd R_1F_1$. Thus $(F_1F_2, GJ) = (D_1D_2, HK)$.
\end{cor}

\begin{proof}
The proof of the theorem shows that the point $D_1J \cd F_1K$ lies on the line $R_1R_2$. Reversing the roles of $R_1$ and $R_2$ gives that $D_2G \cd F_2H$ lies on $R_1R_2$. But this line is the axis of $\pi$, and this implies that $G$ on $AB$ must map to $H$ on $BC$.
\end{proof}

\begin{cor}
\label{cor:Simsonconv}
If the points $D_1, E_1, F_1$ for the ordinary point $R_1$ are collinear, then the point $R_2 = \gamma_P(R_1)$ is an infinite point; i.e., the lines $Ah_a(R_1)$, $Bh_b(R_1)$, and $Ch_c(R_1)$ are parallel.
\end{cor}

\begin{proof}
If both points $R_1$ and $R_2$ are ordinary, then the line $D_1E_1$ would intersect the pedal conic in three points, impossible unless two of the points $D_1, E_1, F_1$ are equal. But the latter is also impossible, since these points lie on three distinct lines through $R_1$. Note that the directions $QD, QE, QF$ are all distinct, because they are conjugate to the directions of the sides $BC, AC, AB$ with respect to the inconic.
\end{proof}

\noindent {\bf Remark.} Theorem \ref{thm:PedalTriangles} generalizes the classical result that the pedal triangles of isogonal conjugates are inscribed in a common circle.  See Proposition \ref{prop:R1R2conic} below.  \medskip

In the following two results (Prop. \ref{prop:Q^2} and Cor. \ref{cor:Q2}), we make very limited use of barycentric coordinates to identify the map $\gamma_P$ with a construction of Dean and Lamoen \cite{dvl}.  We will give a synthetic proof of this connection in another paper.

\begin{prop}
\label{prop:Q^2}
The mapping $\gamma_P$ is a reciprocal conjugation, in the sense of Dean and van Lamoen \cite{dvl}, with pole $P_0 = Q^2$, the point whose barycentric coordinates, with respect to $ABC$, are the squares of the barycentric coordinates of $Q$.
\end{prop}

\begin{figure}
\[\includegraphics[width=4.5in]{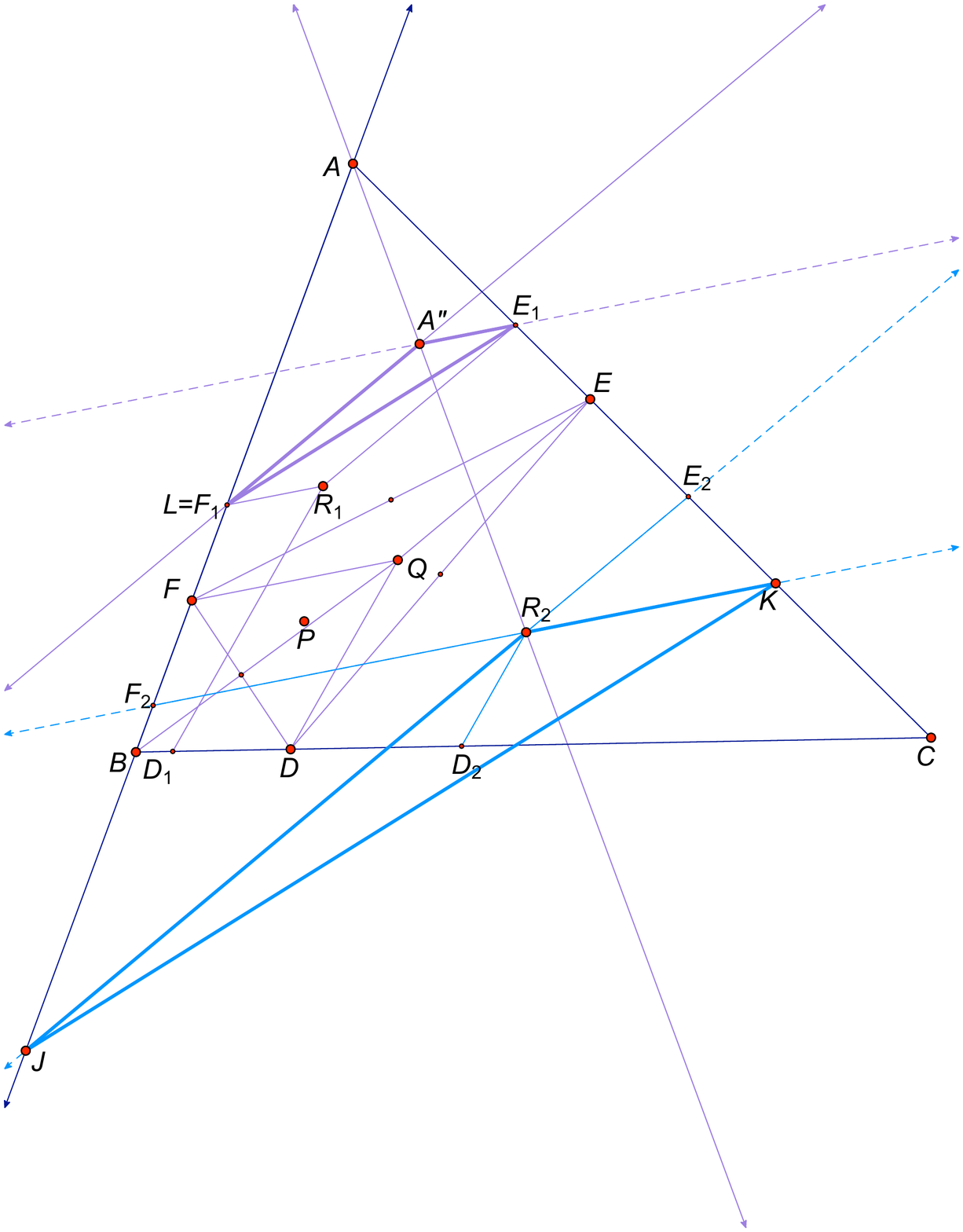}\]
\caption{Proof of Proposition \ref{prop:Q^2}}
\label{fig:2.3}
\end{figure}

\begin{proof}
(See Figure \ref{fig:2.3}.) We use the fact that $R_1D_1F_1 \sim R_2JK$ from the proof of Theorem \ref{thm:PedalTriangles}. Switching to the vertex $A$, this similarity becomes $R_1E_1F_1 \sim R_2JK$, where
\[J = AB \cd R_2E_2, K = AC \cd R_2F_2.\]
In particular, we have that $E_1F_1 \pa JK$ as in the proof of Theorem \ref{thm:PedalTriangles}. Let the line $A''E_1$ parallel to $QF$ intersect the line $AR_2$ in the point $A''$, and let the line $A''L$ parallel to $QE$ intersect $AB$ in the point $L$. Then $A''E_1 \pa R_2K$ and $A''L \pa R_2J$. We also have
\[\frac{A''E_1}{A''L} = \frac{R_2K}{R_2J}\]
from the similarities $AA''E_1 \sim AR_2K$ and $AA''L \sim AR_2J$. It follows that $A''LE_1 \sim R_2JK$, so that $LE_1 \pa JK$ and hence $LE_1 \pa E_1F_1$. This implies that $LE_1 = E_1F_1$ and $L = F_1$ since $L$ lies on $AB$. Thus, $R_1E_1A''F_1$ is a parallelogram, with $A''$ on the line $AR_2 = A\gamma_P(R_1)$. This shows that $R_2 = \gamma_P(R_1)$ can be found using the construction in \cite{dvl}, Prop. 1. By the results of \cite{dvl}, the mapping $\gamma_P$ is given in barycentric coordinates by
\begin{equation}
\label{eqn:5.2}
\gamma_P(r, s, t) = \left(\frac{l}{r}, \frac{m}{s}, \frac{n}{t}\right),
\end{equation}
where $lmn \ne 0$. Since $Q=(r,s,t)$ is a fixed point of $\gamma_P$ we can take $(l, m, n) = (r^2,s^2,t^2)=Q^2$. Thus, $Q^2$ is the pole of the reciprocal conjugation $\gamma_P$.
\end{proof}

\begin{cor}
\label{cor:Q2}
The point $Q^2=\gamma_P(G)$ is the isotomcomplement of the generalized orthocenter $H$ of the point $P$ with respect to $ABC$.
\end{cor}

\begin{proof}
In the construction of the proposition, the lines $QD, QE, QF$ are parallel to the cevians of the generalized orthocenter $H = (f, g, h)$ for $P$. Dean and van Lamoen state that in this case the pole of the corresponding reciprocal conjugation is the point $P_0 = (f(g+h), g(f+h), h(f+g))$. But these are just the coordinates of the isotomcomplement of $H$. This proves the corollary.
\end{proof}

We will now give a synthetic proof of the relationship of the last corollary, in the form $\gamma_P(G)=K \circ \iota(H)$.  This proof uses the relationship $\gamma_P(O)=H$, which we will prove in Proposition \ref{prop:gammaH} below.  (Also see Theorem \ref{thm:persG}.)

\begin{thm}
If the point $H$ is not a vertex of $ABC$ and $P$ does not lie on $\iota(l_\infty)$, then $\gamma_P=\delta_H$.
\label{thm:gammadelta}
\end{thm}

\begin{proof}
First note that if $H$ is on one of the sides of $ABC$, then it has to be a vertex.  If $H$ is on $BC$, say, but is not $B$ or $C$, then $BA = h_b(BC) =
h_b(BH) = BO$, and similarly $CA = h_c(CB) = h_c(CH) = CO$, so $O = A$ and $H =K^{-1}(A)$, which is not on $BC$; contradiction.  Also, $H$ cannot lie on a side of the anticomplementary triangle $K^{-1}(ABC)$.  If it did, then $K(H)=O$ would lie on one of the sides of $ABC$, and since the ordinary point $O$ is the center of the circumconic $\tilde \cC_O$, this would force $O$ to be a midpoint of a side, implying again that $H=K^{-1}(O)$ is a vertex.  \smallskip

From I, Corollary 3.6 we have that $\delta_P(P)=Q'=K(P)$, for points $P$ that are not on the sides of $ABC$ or $K^{-1}(ABC)$.  Hence, $\delta_H(H)=K(H)=O$. We will also use the fact that $\gamma_P(H)=O$ from Proposition \ref{prop:gammaH}.  This shows that the maps $\gamma_P$ and $\delta_H$ share the point pair $(H, O)$.\smallskip

Now consider the involution $\pi_A$ on the line $BC$ given by $\pi_A(D^*) = AT_H(D^*) \cdot BC$, where $T_H$ maps $ABC$ to the cevian triangle of $H$.  This is one of three involutions which are used to define the map $\delta_H$.  (See the discussions following Lemma 3.4 and Theorem 3.5 in Part I.)  It is clear that $\pi_A$ switches $B$ and $C$, and also the pair $D_H = AH \cdot BC$ and $D_O = AO \cdot BC$, since $\delta_H(H)=O$.  Since $H$ is not on $AB$ or $AC$, these are two different pairs of points, so the involution $\pi_A$ is uniquely determined.   \smallskip

We claim that this is the same involution which is induced by $\gamma_P$ on the line $BC$.  This follows from the fact that the involution
$\pi'_A$ given by $\pi'_A(D^*) = Ah_a(D^*) \cdot BC$, where $h_a$ is affine reflection in $AQ$ in the direction of $EF$, also switches the pairs $(B, C)$ and $(D_H, D_O)$, since $\gamma_P(H)=O$.  Hence the two involutions must be the same: $\pi_A = \pi'_A$.  Similar statements are true for the other sides: $\pi_B = \pi'_B$ and $\pi_C = \pi'_C$. But $\delta_H(X)$ is the intersection of $A\pi_A(AX \cdot BC)$, $B\pi_B(BX \cdot AC)$, and $C\pi_C(CX \cdot AB)$.  Since the same is true for the map $\gamma_P(X)$, with $\pi$ replaced by $\pi'=\pi$, we have $\gamma_P(X) = \delta_H(X)$.
\end{proof}

\begin{cor}
If $P$ does not lie on $\iota(l_\infty)$ or any of the conics $\overline{\mathcal{C}}_A, \overline{\mathcal{C}}_B, \overline{\mathcal{C}}_C$ of \cite{mmv} (so that $H \neq A,B,C$), then $\gamma_P(G) = K \circ \iota(H)$ is the isotomcomplement of the generalized orthocenter $H$.
\end{cor}

\begin{proof}
We have $\gamma_P(G) = \delta_H(G) = K \circ \iota(H)$, by I, Corollary 3.6, with $P=H$.
\end{proof}

\end{section}

\begin{section}{Generalized isogonal relationships.}

\begin{prop}
\label{prop:uniqueC}
Let $\psi$ be an involution on the line at infinity, which is induced by the polarity for a central conic, and let $A,B,C$ be non-collinear points.  Then there is a unique conic through $A,B,C$ whose polarity induces $\psi$.
\end{prop}

\noindent {\bf Remark.} We only need the proposition for central conics, since pedal conics are only defined when $P$ is not on the Steiner circumellipse $\iota(l_\infty)$ and $Q$ is ordinary. \smallskip

\begin{proof}
Let $\psi$ be the involution on $l_\infty$ induced by the polarity for a conic $\C$. If $L$ and $M$ are the midpoints of $AB$ and $AC$, and the points at infinity on $AB$ and $AC$ are $R$ and $S$, then the lines $L\psi(R)$ and $M\psi(S)$ are distinct diameters of $\C$, since otherwise $LM=\psi(R) \psi(S)$ would the line at infinity.  Thus, their intersection is the center of $\C$ (an ordinary point because $\psi(R) \neq \psi(S)$).  If the center is not collinear with any two of the vertices of $ABC$, the conic is determined by the vertices and two of their reflections through the center, which makes $5$ points.  If the center is collinear with two of the points, say $B$ and $C$, then it has to be the midpoint of $BC$.  If the point at infinity on $BC$ is $T$, the tangent lines at $B$ and $C$ are $B\psi(T)$ and $C\psi(T)$.  Then $\C$ is determined by the $3$ vertices and the two tangent lines.  See \cite{co2}.
\end{proof}

\begin{prop}
\label{prop:Pedalinv}
If the point $P$ does not lie on $\iota(l_\infty)$, the pedal conic $\mathcal{P}$ of an ordinary point $R_1$ and its ordinary $P$-isogonal conjugate $R_2$ induces the same involution $\psi$ on $l_\infty$ as the inconic $\I$.   The center of $\mathcal{P}$ is the midpoint $M$ of the segment $R_1R_2$.
\end{prop}

\begin{proof}
(See Figure \ref{fig:3.1}.)  Let $R_2$ be any point on the line $l=Ah_a(R_1)$ and $M$ the midpoint of segment $R_1R_2$.  Let $E_1,F_1$ and $E_2,F_2$ be the intersections with $AC$ and $AB$ of the lines through $R_1$ and $R_2$ parallel to $QE, QF$.  We show that the conic $\C_M$ on the points $E_1, F_1, F_2$ with center $M$ induces $\psi$.  If $R_1$ is fixed, then as $R_2$ varies on the line $l$, the point $M$ varies on a line $l'$ parallel to $l$.  The line $l'$ lies on the midpoint of $E_1F_1$, since $M$ is equal to this midpoint when $R_2$ is the point $A''$ in the proof of Proposition \ref{prop:Q^2}; for $R_1E_1A''F_1$ is a parallelogram with diagonals $R_1A''=R_1R_2$ and $E_1F_1$ (see Figure \ref{fig:2.3}).  Hence, the direction of the line $l'$ is conjugate to the direction of the line $E_1F_1$ for any of the conics $\C_M$, since $E_1F_1$ is a chord on each of these conics.  Furthermore, if $F_3$ is the midpoint of segment $F_1F_2$, then it is easy to see that $MF_3 \pa QF$.  Hence, the directions of $AB$ ($=F_1F_2$) and $QF$ are conjugate for any conic $\C_M$.  (If $R_1$ lies on $AQ$ and $R_2=R_1$, then if $\C_M$ is the conic with center $M=R_1$ which is tangent to $AB$ at $F_1=F_2$ and tangent to $AC$ at $E_1=E_2$, the same conclusions hold.)  Since the involutions induced on $l_\infty$ by these conics share two pairs of conjugate points, they must all be the same involution $\psi_1$. \medskip

\begin{figure}
\[\includegraphics[width=4.5in]{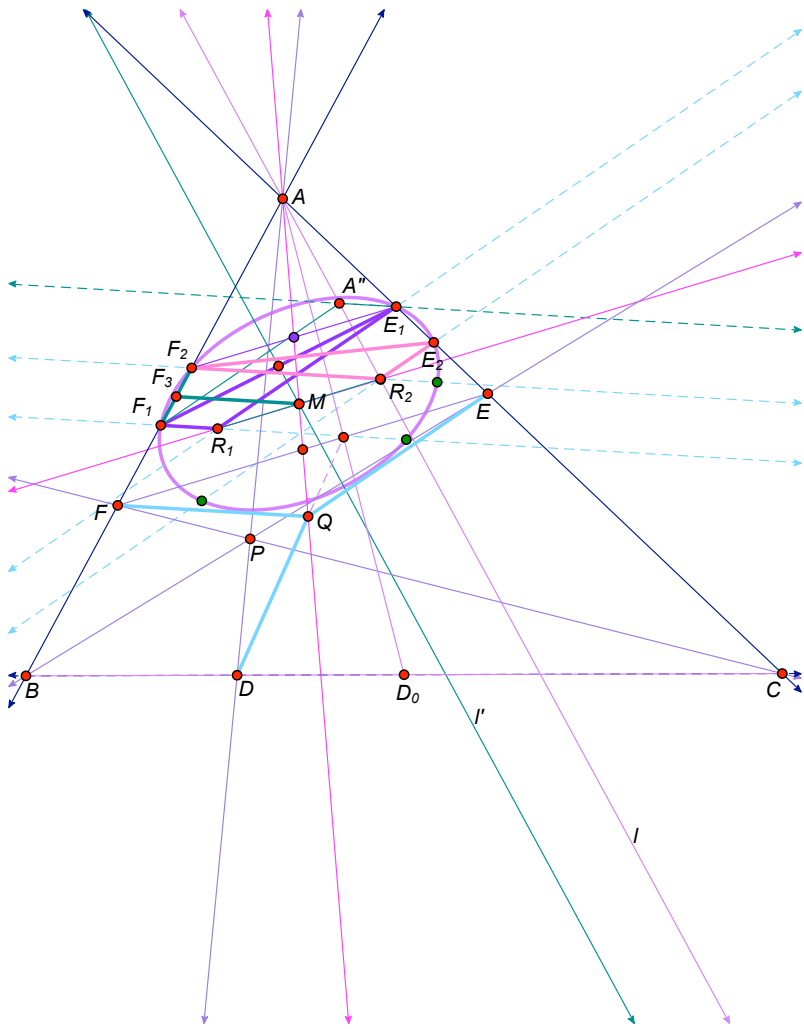}\]
\caption{Proof of Proposition \ref{prop:Pedalinv}}
\label{fig:3.1}
\end{figure}

We now show this involution coincides with $\psi$.  To do this, consider the point $R_2=h_a(R_1)$ and its corresponding midpoint $M$.  With this choice of $R_2$, $M$ lies on the line $AQ$, by definition of the affine reflection $h_a$.  Now the map $h_a$ interchanges the directions of $QE$ and $QF$, since $AQ$ lies on the midpoint of $EF$ (I, Theorem 2.4), so it maps triangle $E_1R_1F_1$ to $F_2R_2E_2$.  This implies that $F_2=h_a(E_1)$ and $E_2=h_a(F_1)$.  Thus $E_1F_2 \pa EF$, and the midpoint of $E_1F_2$ lies on $AQ=AM$.  Hence, $EF$ and $AQ$ represent conjugate directions for both involutions $\psi_1$ and $\psi$ ($Q$ is the center of $\I$, lying on $E$ and $F$).  But we have already shown that $\psi_1$ has the conjugate pair of directions $AB$ and $QF$, and this pair is shared by $\psi$, since $AB$ is tangent to the conic $\I$ at $F$.  This proves that $\psi_1=\psi$.  \medskip

Once again, let $R_2$ be any point on the line $Ah_a(R_1)$.  If $E_3$ is the midpoint of $E_1E_2$, then $ME_3 \pa QE$.  It follows from what we have proved that $ME_3$ and $E_1E_2$ ($=AC$) are conjugate directions for the conic $\C_M$, which implies that $E_2$ must also lie on $\C_M$.  This shows that there is a unique conic on the points $E_1,F_1,E_2,F_2$ which induces the involution $\psi$ and whose center is the midpoint $M$ of $R_1R_2$.  Now let $R_2=\gamma_P(R_1)$.  By the above argument, there are conics $\C_1,\C_2,\C_3$, all with center $M$, and all inducing the same involution $\psi$, lying on the point sets $\{E_1,F_1,E_2,F_2\}, \{D_1,F_1,D_2,F_2\}$, and $\{D_1,E_1,D_2,E_2\}$, respectively.  However, any two of these conics share four points and induce the same involution on $l_\infty$.  For example, $\C_1 \cap \C_2$ contains $\{F_1,F_2,F_1',F_2'\}$, where $F_i'$ is the reflection of $F_i$ through $M$.  By Proposition \ref{prop:uniqueC}, these conics must all be the same conic, the pedal conic of $R_1$ and $R_2$.  This completes the proof.
\end{proof}

\begin{cor}
The maps on the line $l_\infty$ induced by $h_a,h_b, h_c$ commute with the involution $\psi$.
\end{cor}

\begin{proof}
For example, $h_a \psi h_a^{-1}=\psi$ follows from the fact that when $R_2=h_a(R_1)$, the map $h_a$ takes the point set $\{E_1,F_1,E_2,F_2\}$ to itself and fixes $M$.  Therefore, $h_a(\C_M)=\C_M$.  Alternatively, the center $EF \cdot l_\infty$ of the homology $h_a$ is the pole of its axis $AQ$, with respect to the conic $\I$, and therefore $h_a$ maps $\I$ to itself.  See \cite{co2}, p. 76, Exer. 4.
\end{proof}

\begin{lem}
\label{lem:mapl}
If $l$ is any line not lying on  a vertex of triangle $ABC$, then $\gamma_P(l)$ is a circumconic for the triangle.
\end{lem}

\noindent {\bf Remark.}  In particular, this lemma also holds for the isotomic map $\iota = \gamma_G$. \smallskip

\begin{proof}
The image $\gamma_P(l)$ is the locus of points $h_a(AR) \cdot h_b(BR)$, for points $R$ on $l$.  Since $h_a$ is a projective collineation, we have that
\[h_a(AR) \ \barwedge  \ AR \ \stackrel{l}{\doublebarwedge}  \ BR \ \barwedge  \ h_b(BR).\]
($\barwedge$ denotes a projectivity.)  Thus, the locus $h_a(AR) \cdot h_b(BR)$ is either a line or a conic.  However, if $R$ lies on a side of $ABC$, then $\gamma_P(R)$ is the opposite vertex.  Thus, $\gamma_P(l)$ lies on the vertices of $ABC$ and is therefore a conic.
\end{proof}

\begin{prop}
\label{prop:gammaCO}
For any point $P$ not on the sides of $ABC$ or $K^{-1}(ABC)$, we have that $\gamma_P(l_\infty)=\tilde{\C}_O$.
\end{prop}

\begin{proof}
If $P=G$, then $P'=Q=G$, $\gamma_P=\iota$ is the isotomic map, and $O=T_G^{-1}(K(G))=K^{-1}(K(G))=G$, so the conic $\tilde \cC_O=\iota(l_\infty)$ is the Steiner circumellipse.  Now assume $P \neq G$.  From the previous lemma we know that $\gamma_P(l_\infty)$ is a circumconic of $ABC$.  We will show that $\gamma_P(l_\infty)$ lies on the midpoints of the sides of the anticevian triangle of $Q$ with respect to $ABC$, which is $Q_aQ_bQ_c=T_{P'}^{-1}(ABC)$.  This will show that $\gamma_P(l_\infty)$ has six points in common with $\tilde{\C}_O$, by III, Theorem 2.4.  It suffices to show that $\gamma_P(EF \cdot l_\infty)$ is the midpoint of $Q_bQ_c$.  Let $Y=EF \cdot l_\infty$.  We know that the map $h_a$ fixes the line $AY=Q_bQ_c$, since this line lies on $A$ and is parallel to $EF$ (I, Theorem 3.9).  Furthermore, $h_b$ takes the line $BY$ to its harmonic conjugate with respect to $BQ$ and $B(DF \cdot l_\infty)$.  The section of this harmonic set of lines by $AY=Q_bQ_c$ is a harmonic set of points.  Now $BQ$ lies on $Q_b$, while $B(DF \cdot l_\infty)$ lies on $Q_c$.  Since $BY \pa Q_bQ_c$, it follows that $h_b(BY)$ intersects $Q_bQ_c$ at the midpoint of $Q_bQ_c$, which coincides with the intersection $AY \cdot h_b(BY)=Ah_a(Y) \cdot B h_b(Y)=\gamma_P(Y)$.  A similar argument shows that $\gamma_P(DF \cdot l_\infty)$ is the midpoint of $Q_aQ_c$ and $\gamma_P(DE \cdot l_\infty)$ is the midpoint of $Q_aQ_b$.  To finish the proof we just have to check that at least two of these midpoints do not coincide with the vertices $A,B$, or $C$.  If $A$ is the midpoint of $Q_bQ_c$, then $T_{P'}(A)=D_3$ is the midpoint of $T_{P'}(Q_bQ_c)=BC$ (I, Corollary 3.11), which implies that the point $P'$ lies on the median of $ABC$ through $A$.  If two midpoints coincide with vertices, then $P=P'=G$ is the centroid.
\end{proof}

Combined with Corollary \ref{cor:Simsonconv}, Proposition \ref{prop:gammaCO} shows that any point $R_1$ for which
\[R_1D_1 \pa QD, \ R_1E_1 \pa QE, \ R_1F_1 \pa QF,\]
where $D_1,E_1,F_1$ are collinear points lying on the respective sides $BC, AC, AB$, must lie on the circumconic $\tilde{\C}_O$, since $R_2=\gamma_P(R_1)$ lies on $l_\infty$.  This is just the converse of the generalized Simson theorem.  See \cite{ac}, p. 140.  To prove Simson's theorem in this situation we first prove the following property.  (Cf. Prop. 288 in \cite{ac}.)

\begin{prop}
\label{prop:Coinv}
Assume $P$ does not lie on $\iota(l_\infty)$.  Given the point $R_1$ on the circumconic $\tilde{\C}_O$, let $D_1,E_1,F_1$ denote the `feet' of the parallels dropped to $BC,AC,AB$ from $R_1$ in the directions $QD, QE, QF$.  Let $A', B', C'$ be the second intersections of the lines $R_1D_1, R_1E_1, R_1F_1$ with $\tilde{\C}_O$.  Then the lines $AA', BB', CC'$ are parallel; i.e., triangles $ABC$ and $A'B'C'$ are perspective from a point on $l_\infty$.
\end{prop}

\begin{figure}
\[\includegraphics[width=4.5in]{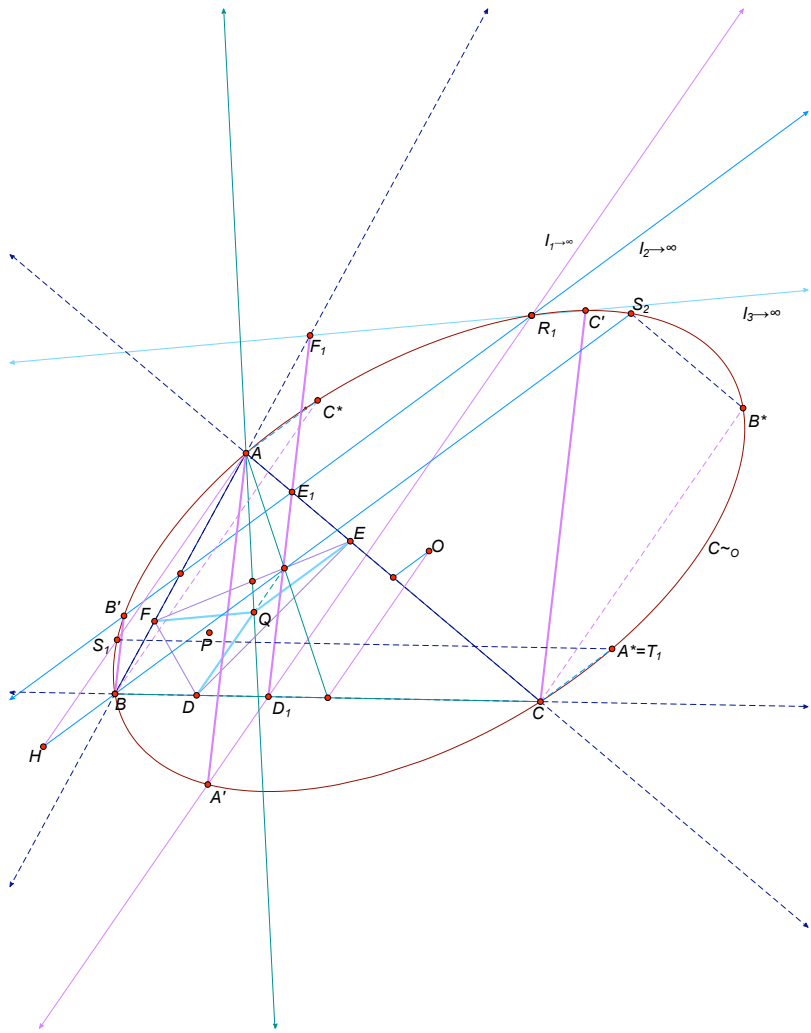}\]
\caption{Proof of Proposition \ref{prop:Coinv} and Theorem \ref{thm:Sim}}
\label{fig:3.2}
\end{figure}

\begin{proof}
(See Figure \ref{fig:3.2}.)  Let $I_1$ and $I_2$ denote the intersections of the lines $QD, QE$ with $l_\infty$, and consider the involution on the conic $\tilde{\C}_O$ given by intersecting it with secant lines through $I_1$.  Then $R_1$ maps to the second intersection $A'$ of the conic with $R_1D_1$, and $A$ maps to the second intersection of the conic with $AH$, where $H$ is the generalized orthocenter for $P$, which is $S_1$, in the notation of III, Proposition 3.2, since $AH \pa QD$.  Furthermore, $B$ maps to the reflection $\textsf{R}_O(C)$ of $C$ in $O$, which we denote by $C^*$.  This is because the direction of $BC^*$ is conjugate (with respect to $\tilde{\C}_O$) to the direction of $BC$, since $CC^*$ is a diameter, and therefore $BC^* \pa QD$.  In the same way, $C$ maps to $B^*=\textsf{R}_O(B)$.  Thus we have the involution
\[R_1ABC \ \barwedge  \ A'S_1C^*B^*.\]
There is also an involution on $\tilde{\C}_O$ given by intersecting it with secant lines through $I_2$, for which we have
\[R_1ABC \ \barwedge  \ B'C^*S_2A^*,\]
where $S_2$ is the second intersection of $BH$ with the conic and $A^*=\textsf{R}_O(A)$.  This gives the projectivity
\[B'C^*S_2A^* \ \barwedge  \ A'S_1C^*B^*.\]
If the axis of this projectivity is $l$, then $S_1A^* \cdot C^*B^*$ and $S_2B^* \cdot C^*A^*$ lie on $l$.  But both of these points lie on $l_\infty$, by III, Proposition 3.2 and III, Corollary 3.7 (where $A^*$ is denoted by $T_1$), since $C^*B^* \pa CB \pa S_1A^*$ and $C^*A^* \pa CA \pa S_2B^*$.  Hence, $l=l_\infty$.  It follows from the projectivity that $A^*A' \pa B^*B'$.  In the same way, using the product of two involutions on $\tilde{\C}_O$ defined by taking secant lines through the ideal points $I_1$ and $I_3$ (lying on $QF$), we see that $A^*A' \pa C^*C'$.  Now the lines $AA', BB', CC'$ lie in conjugate directions to the parallel lines $A^*A', B^*B', C^*C'$, and are therefore parallel to each other.  This proves the proposition.
\end{proof}

\begin{thm}{(Simson's Theorem)}
\label{thm:Sim}
Assume $P$ does not lie on $\iota(l_\infty)$.  If $R_1$ lies on the circumconic $\tilde{\C}_O$, then the `feet' $D_1, E_1, F_1$ of the parallels dropped to $BC, AC, AB$ from $R_1$ in the directions $QD, QE, QF$, are collinear.
\end{thm}
 \begin{proof}
From Proposition \ref{prop:Coinv} we have that $A, B, C, A', B', C', R_1$ are points on a conic with $ABC$ perspective to $A'B'C'$ from some point $O_1$.  Looking at the hexagon $A'R_1B'BCA$, Pascal's theorem implies that $R_1A' \cdot BC = D_1, R_1B' \cdot CA=E_1$, and $AA' \cdot BB'=O_1$ are collinear.  Similarly, the hexagon $A'R_1C'CBA$ gives that $R_1A' \cdot BC = D_1, R_1C' \cdot AB=F_1$, and $AA' \cdot CC'=O_1$ are collinear.  Hence, the points $D_1, E_1, F_1$ lie on the line $D_1O_1$, which is parallel to $AA'$.
\end{proof}

\begin{prop}
\label{prop:R1R2conic}
If the pedal triangles of two ordinary points $R_1$ and $R_2$ are inscribed in a common conic $\mathcal{P}$, inducing the same involution on $l_\infty$ as $\tilde{\C}_O$, then $R_2=\gamma_P(R_1)$, and $\mathcal{P}$ is the pedal conic for $R_1$ and $R_2$.
\end{prop}

\begin{proof}
First note that, with the notation from Theorem \ref{thm:PedalTriangles} (see Figure \ref{fig:2.2}), the points $D_1, E_1, F_1$ are not collinear, since they lie on the conic $\mathcal{P}$.  Theorem \ref{thm:Sim} shows that $R_1$ does not lie on the conic $\tilde{\C}_O$, and therefore the point $\gamma_P(R_1)$ is ordinary, by Proposition \ref{prop:gammaCO}.  Now Theorem \ref{thm:PedalTriangles} shows that the pedal triangles of $R_1$ and $\gamma_P(R_1)$ are inscribed in their common pedal conic $\mathcal{P}'$, which induces the same involution on $l_\infty$ as the given conic $\mathcal{P}$ (Proposition \ref{prop:Pedalinv}).  By Proposition \ref{prop:uniqueC} , $\mathcal{P}=\mathcal{P}'$, and since the intersections of $\mathcal{P}$ with the sides of $ABC$ determine the pedal triangle of $R_2$, it follows that $R_2=\gamma_P(R_1)$.
\end{proof}

\noindent {\bf Remark.} If $\mathcal{P}$ is the unique conic lying on the vertices of the pedal triangle of an ordinary point $R_1 \notin \tilde \cC_O$ inducing the same involution $\psi$ on $l_\infty$ as in  Proposition \ref{prop:Pedalinv}, then the other intersections of $\mathcal{P}$ with the sides of $ABC$ are precisely the vertices of the pedal triangle for $\gamma_P(R_1)$.  This follows in the same way as in the proof just given. \medskip

The next proposition shows that the generalized orthocenter $H$ and generalized circumcenter $O$ stand in the same relationship as their classical counterparts, if the ordinary isogonal map $\gamma$ is replaced by $\gamma_P$.

\begin{prop}
\label{prop:gammaH}
For any point $P$ not on the sides of $ABC$ or $K^{-1}(ABC)$, we have that $\gamma_P(O)=H$.
\end{prop}

\begin{proof}
(See Figure \ref{fig:3.2}.)  First assume $P$ does not lie on the Steiner circumellipse $\iota(l_\infty)$.  If the point $H$ does not coincide with a vertex, the nine-point conic $\Nh$ for the quadrangle $ABCH$ is a conic on the vertices of the pedal triangles of the ordinary points $H$ and $O$, and induces the same involution on $l_\infty$ as $\tilde{\C}_O$, by III, Proposition 3.6.  The result follows from Proposition \ref{prop:R1R2conic}.  If $H$ coincides with the vertex $A$, for example, then $O=D_0$ is a point on the opposite side $BC$, so $\gamma_P(O)=A=H$.  If $P$ lies on $\iota(l_\infty)$, then the points $O$ and $H$ both coincide with $Q=P'$ (III, Corollary 2.3), which is fixed by $\gamma_P$.
\end{proof}

For the proof of the next result, recall that the point $X'$ is the center of the map $\mathcal{S}'=T_{P'} \circ T_P$.

\begin{prop}
\label{prop:OQ}
If the ordinary point $P$ does not lie on a median of $ABC$, then $\gamma_P(SQ)=\Cp$, and the tangent to $\Cp$ at $Q$ is $SQ$.  If $P$ does not lie on $\iota(l_\infty)$, then the line $SQ=OQ$.  In particular, the pole of the line $QQ'$ with respect to the conic $\Cp$ is $S$.
\end{prop}

\begin{proof}
First assume $P$ does not lie on $\iota(l_\infty)$.  By III, Theorem 3.9 we know that the point $S$ lies on $OQ$, and we claim that $S \neq Q$, so that $OQ=SQ$.  If $S = Q$, then since $S$ is the center of the map $\textsf{M}$ and $\textsf{M}(O) = Q$, where $Q$ and $O$ are ordinary points, we conclude that $O=S=Q$, yielding $T_{P'}(O) = T_{P'}(Q) = P'$ (I, Theorem 3.7).  On the other hand, $K(Q)=T_{P'}(O)=P'$ by the affine formula for $O$, so $K(Q) = P' = K^{-1}(Q)$, giving that $Q = G$ and $P$ lies on a median, which is contrary to hypothesis.  Hence, $OQ=SQ$.  We next check that the line $OQ$ does not lie on a vertex of $ABC$.  If $OQ$ lies on $C$, for example, then $T_{P'}(OQ)=K(Q)P'$ lies on the point $T_{P'}(C)=F_3$, i.e. $K(Q)P'=CP'$.  But $K(Q)$ is the midpoint of $P'Q$, so $Q$ and also $G$ lies on $CP'$, so that $P'$ lies on the median $CG$.  Since $S=\gamma_P(P)$ (Proposition \ref{prop:PS}) never coincides with a vertex of $ABC$ ($P$ never lies on a side of $ABC$), Lemma \ref{lem:mapl} shows that $\gamma_P(OQ)$ is a circumconic of $ABC$ lying on the points $\gamma_P(S)=P$ and $\gamma_P(Q)=Q$.  Since $P \neq Q$, this shows that $\gamma_P(OQ)=ABCPQ=\Cp$.  To show $SQ$ is tangent to $\Cp$, argue as in the proof of \cite{mmv}, Proposition 2.4.  If the point $L$ lies on $SQ \cap \Cp$, then $L$ and $\gamma_P(L)$ are in this intersection, so either $L=Q$ or $L=\gamma_P(L)$.  But the only fixed points of $\gamma_P$ are $Q$ and the vertices of the anticevian triangle of $Q$ with respect to $ABC$.  If $L$ were one of these vertices, then $SQ=QL$ would lie on a vertex of $ABC$, which we showed above to be impossible.  Therefore, $L=Q$ is the only point in $SQ \cap \Cp$, showing that $SQ$ is the tangent line at $Q$. \smallskip

Now assume $P$ does lie on $\iota(l_\infty)$.  Then $Q=P'=O \in l_\infty$ (see III, Corollary 2.3).  In this case, the map $\mathcal{S}=T_PT_{P'}=K^{-1}$ (I, Theorem 3.14) has center $X=G$.  From I, Corollary 3.11, we know that $X'=T_{P'}(X)=T_{P'}(G)=G_2$ is an ordinary point.  Since $X'$ is a fixed point of $\mathcal{S}'=T_{P'}T_P$, we have that
$$\textsf{M}(X')=T_{P'}K^{-1}T_P(X')=(T_{P'}T_P)^2(X')=X',$$
so $X'=G_2=S$ is the center of the map $\textsf{M}$, using the fact that $\textsf{M}$ is a translation or homothety with a unique ordinary fixed point (III, Theorem 3.4).  By II, Lemma 2.5, we also know that $G$ is the midpoint of segment $G_1G_2$, where $G_1=T_P(G)$, and $GG_1 \pa PP'$.  But $P'=Q$ in this case, so the infinite point $Q$ lies on the line $GG_2=GS$.  Then II, Theorem 4.3 shows that $GG_1=GS=SQ$ is an asymptote of $\Cp$; in other words, $SQ$ is the tangent to $\Cp$ at $Q$.  Since $P$ is not on a median of $ABC$, none of the vertices lie on the line $SQ=GS=GQ$, so the same argument as in the first paragraph of the proof shows that $\gamma_P(SQ)=\Cp$.  The last assertion of the proposition follows from the fact that the point $S=GV \cdot OQ = GV \cdot O'Q'$ is symmetric with respect to $P$ and $P'$, because it is the center of the map $\textsf{M}=T_P \circ K^{-1} \circ T_{P'}$, which is symmetric in $P$ and $P'$ (by III, Proposition 3.12b and Lemma \ref{lem:commute} in the Appendix).  Hence, $SQ'$ is the tangent to $\Cp=\mathcal{C}_{P'}$ at $Q'$.  
\end{proof} 

\begin{thm} Assume that the point $P$ is ordinary and does not lie on a median of $ABC$.
\begin{enumerate}[1.]
\item $X = PQ' \cdot SQ$ ($=PQ' \cdot OQ$, if $P$ does not lie on $\iota(l_\infty)$).
\item The following nine points are always collinear:
\begin{equation*}
X, \ T_P(P'), \ T_P(G), \ Q, \ S, \ O, \ \textsf{M}(Q), \ T_{P'}^{-1}(G), \ and \ T_{P'}^{-1}(Q)=T_P^{-1}(H).
\end{equation*}
\item $T_P(G)$ is the pole of $PQ$ with respect to the conic $\cC_P$, and the polar of $P$ is $p=PT_P(G)$. 
\item $\textsf{M}(Q)$ is the pole of $P'Q$ with respect to $\cC_P$.
\item The tangent to $T_{P'}(\cC_P)$ at $P'$ is $P'Q$.
\item The pole of $QQ'$ with respect to $T_{P'}^{-1}(\cC_P) = T_P^{-1}(\cC_P)$ is $G$.
\item $\textsf{M}(QQ') = K^{-1}(PP')$.
\item The tangent to $\cC_P$ at $H$ is $h=HT_{P}(P')$.
\end{enumerate}
\end{thm}

\noindent {\bf Remark.}  The point $T_{P'}^{-1}(G)$ in the second statement is the centroid of the anticevian triangle of $Q$, and the point $G_1=T_P(G)$ is the centroid of the cevian triangle of $P$.  Also, $\textsf{M}(Q)=T_{P'}(P')$.  \smallskip

\begin{proof} 
By applying $T_P$ to $P'GQ$, we see that $T_P(P'), T_P(G), Q$ are collinear. By applying $T_{P'}^{-1}$ to $P'GQ$, we see that $Q, T_{P'}^{-1}(G), T_{P'}^{-1}(Q)=T_P^{-1}(H)=\tilde H$ are collinear (see III, equation (3)).  Now applying $\textsf{M}=T_P \circ K^{-1} \circ T_{P'}$ to $Q, T_{P'}^{-1}(G), T_{P'}^{-1}(Q)$, we see that $\textsf{M}(Q), T_P(G), T_P(P')$ are collinear.  But these are collinear with $Q$, since $T_P(Q)=Q$!  It follows that $T_P(G)$ and $T_P(P')$ lie on the line $QM(Q) = SQ$.  Applying $\textsf{M}^{-1}=T_{P'}^{-1} \circ K \circ T_P^{-1}$, which fixes this line ($S$ is the center of $\textsf{M}$), we see that $T_{P'}^{-1}(G)$ and $T_{P'}^{-1}(Q)$ lie on the same line. Thus, the points
$$T_P(P'), T_P(G), Q, S, O=\textsf{M}^{-1}(Q), \textsf{M}(Q), T_{P'}^{-1}(G), \ \textrm{and} \ T_{P'}^{-1}(Q)$$
are collinear. This implies that $\mathcal{S}'=T_{P'} \circ T_P$ fixes the line $P'Q$ since $T_P(P'Q) = T_P(P')Q=T_{P'}^{-1}(Q)T_{P'}^{-1}(P')=T_{P'}^{-1}(P'Q)$, so the center of $\mathcal{S}'$, namely $X'$, is on $P'Q$.  Applying the map $\eta$ shows that $X$ is on $PQ'$. Finally, $\mathcal{S}(Q) = T_P \circ T_P'(Q) = T_P(P')$ lies on the line $SQ$, so $X$ is on $SQ$ as well, proving parts (1.) and (2.). \smallskip

With respect to the conic $\cC_P$, the pole $p \cdot q$ of $PQ$ lies on $q = SQ$, by Proposition \ref{prop:OQ}.  Also, since $V$ lies on $PQ$, $p \cdot q$ lies on $v = GV_\infty$ (see II, p. 26). Thus, $p \cdot q$ is the intersection of $GV_\infty$ and $SQ$.  But we already know $T_P(G)$ lies on $SQ$ and $T_P(G)G=G_1G =GV_\infty\pa PP'$ by II, Lemma 2.5; hence, $T_P(G)$ is this intersection, giving (3.).  This implies $PT_P(G)=p$ is the polar of $P$.  \smallskip

Now, $\textsf{M}(Q)$ lies on $SQ = q$, so to prove (4.) we just need to show that $\textsf{M}(Q)$ lies on $p'$, the polar of $P'$.  But $p' = P'T_{P'}(G)$, by (3.) applied to the point $P'$.  Hence, applying $T_{P'}$ to the collinear points $Q, G, P'$, we see that $\textsf{M}(Q) = T_{P'} \circ K^{-1} \circ T_P(Q)=T_{P'}(P')$ lies on $p'$.   \smallskip

Part (5.) of the theorem follows from the fact that the tangent to $\cC_P$ at $Q$, namely $SQ$, lies on $T_{P'}^{-1}(Q)$.  Thus, the tangent to $T_{P'}(\cC_P)$ at $T_{P'}(Q) = P'$ lies on $Q$.  Therefore, this tangent is $P'Q$. \smallskip

Now the tangent to $\cC_P$ at $Q$ goes through $T_P(P')$ so the tangent to $T_P^{-1}(\cC_P)$ at $T_P^{-1}(Q) = Q$ goes through $P'$, i.e. equals $P'Q=QG$. Similarly, since $T_P^{-1}(\cC_P) = T_{P'}^{-1}(\cC_P)$, the tangent at $Q'$ is $Q'G$, so the pole of $QQ'$ with respect to this conic is $G$, giving (6.). \smallskip

For (7.), $P'Q$ goes through $G$, whose polar with respect to $\cC_P$ is $VV_\infty$. Thus, the pole $\textsf{M}(Q)$ of $P'Q$ lies on $VV_\infty$ = $K^{-1}(PP')$ (see II, Proposition 2.3(e)). Similarly, $\textsf{M}(Q')$ lies on $K^{-1}(PP')$, hence $\textsf{M}(QQ') = K^{-1}(PP')$.

Finally, the tangent at $\tilde H=T_P^{-1}(H) = T_{P'}^{-1}(Q)$ to $T_P^{-1}(\cC_P)$ is $T_{P'}^{-1}(SQ)$, since $SQ$ is tangent to $\cC_P$ at $Q$.  The point $P'$ lies on $T_{P'}^{-1}(SQ)$ since $T_{P'}(P') = \textsf{M}(Q)$ lies on $SQ$. Therefore, $T_{P}(P')$ lies on the tangent to $\cC_P$ at $H$.  Note that $\tilde H$ is the midpoint of the segment joining $P'$ and $K^{-1}(H)$, by III, Lemma 3.8, so $T_P(P') \neq H$.  This proves (8.).
\end{proof}

\end{section}

\begin{section}{Circumcevian triangles and the TCC perspector.}

In this section we prove our main theorem.

\begin{defn}  The {\bf circumcevian triangle} of a point $R$ with respect to $ABC$ and the circumconic $\tilde{\C}_O$ is the triangle $A'B'C'$, where $AR, BR, CR$ intersect $\tilde{\C}_O$ in the respective second points $A', B', C'$.
\end{defn}

\begin{prop}
\label{prop:circum}
Assume that $P$ is an ordinary point, not lying on $\iota(l_\infty)$.
\begin{enumerate}[a)]
\item The circumcevian triangle of $Q$ with respect to $ABC$ and $\tilde{\C}_O$ is the triangle
\[A'B'C'=T_{P'}^{-1}(R_1'R_2'R_3'),\]
where $R_1', R_2', R_3'$ are the midpoints of the respective segments $AP', BP', CP'$.
\item The triangle $A'B'C'$ is perspective to the medial triangle $D_0E_0F_0$ from the point $O$.
\item The antipodal triangle of $A'B'C'$ on the conic $\ \tilde{\C}_O$ is the triangle $T_{P'}^{-1}(D_0E_0F_0)$, the medial triangle of the anticevian triangle for $Q$.
\item $A'B'C'$ is homothetic or congruent to the cevian triangle $DEF$ of $P$.
\end{enumerate}
\end{prop}

\begin{proof}
From \cite{mmq}, Corollary 5(b) we know that $D_0, R_1', A_0'$, and $K(Q)$ are collinear (assuming $P'$ is ordinary).  Applying the map $T_{P'}^{-1}$ gives that
\begin{equation}
T_{P'}^{-1}(D_0), \ T_{P'}^{-1}(R_1'), \ T_{P'}^{-1}(A_0') = D_0, \ T_{P'}^{-1}K(Q)=O
\end{equation}
are collinear.  Furthermore, $T_{P'}^{-1}(R_1')$ is the midpoint of the segment $T_{P'}^{-1}(AP')=Q_aQ$, so $T_{P'}^{-1}(R_1')$ lies on the line $AQ$.  Also, $R_1'$ lies on the nine-point conic $\Npp$ of the quadrangle $ABCP'$, so by III, Theorem 2.4, $T_{P'}^{-1}(R_1')$ lies on the conic $\tilde{\C}_O$.  \smallskip

Now there are several cases to consider.  Suppose that $T_{P'}^{-1}(R_1') \neq A$, or equivalently, $R_1' \neq T_{P'}(A) = D_3$.  Then $T_{P'}^{-1}(R_1')$ is the second intersection of $AQ$ with $\tilde{\C}_O$.  This proves a) in the case that none of the midpoints $R_1', R_2', R_3'$ coincides with the respective points $D_3, E_3, F_3$.  \smallskip

Note that if $T_{P'}^{-1}(R_1') = A$, then $R_1'=D_3$ is collinear with $D_0$ and $K(Q)$, so $Q$ lies on the side $K^{-1}(D_0D_3)=K^{-1}(BC)$ of the anticomplementary triangle of $ABC$, as long as $D_0 \neq D_3$.  If $R_1'=D_3=D_0$, then it is easy to see that $P'=K^{-1}(A)$ is a vertex of the anticomplementary triangle $K^{-1}(ABC)$ and $P=P'$, which is excluded.  Thus, at most one of the points $R_1', R_2', R_3'$ can coincide with their counterparts $D_3, E_3, F_3$.  Suppose that $R_1'=D_3$.  As the center of the conic $\Npp$ lying on the points $D_3$ and $D_0$, $K(Q)$ is the midpoint of $D_0D_3$, since we know it is an ordinary point on the line $BC$.  Therefore $O=T_{P'}^{-1}(K(Q))$ is the midpoint of $T_{P'}^{-1}(D_3D_0)=AT_{P'}^{-1}(D_0)$, implying that $\tilde{A}=T_{P'}^{-1}(D_0)$ is the reflection of $A$ in the point $O$, lying on $\tilde{\C}_O$.  Now III, Corollary 3.5 says that the tangent $l$ to $\tilde{\C}_O$ at $\tilde{A}$ is parallel to $BC$.  Since $A$ and $\tilde{A}$ are opposite points on $\tilde{\C}_O$, the tangent line to $\tilde{\C}_O$ at $A$ is also parallel to $BC$, so this tangent is $K^{-1}(BC)=AQ$.  Thus, in this case, $AQ$ only intersects $\tilde{\C}_O$ in the point $T_{P'}^{-1}(R_1')=A$, and $T_{P'}^{-1}(R_1'R_2'R_3')$ is again the circumcevian triangle of $Q$.\smallskip

Part b) follows from the observation that $T_{P'}^{-1}(R_1'), D_0$, and $O$ are collinear, with similar statements for the other vertices.   
Part c) follows from (3) and the fact that $D_0, E_0, F_0$ lie on the nine-point conic $\Npp$ (with respect to $l_\infty$) of the quadrangle $ABCP'$, so the points $T_{P'}^{-1}(D_0), T_{P'}^{-1}(E_0), T_{P'}^{-1}(F_0)$, which are the midpoints of the sides of the anticevian triangle of $Q$, lie on $\tilde{\C}_O=T_{P'}^{-1}(\Npp)$.  As above, the points $R_1', R_2', R_3'$ are distinct from the points $D_0,E_0,F_0$; if $R_1'=D_0$, for example, then $R_1'=D_0=D_3$.  Therefore, $T_{P'}^{-1}(D_0E_0F_0)$ is the antipodal triangle of $A'B'C'$ on $\tilde{\C}_O$. Denoting the half-turn about $O$ by $\textsf{R}_O$, we have
\begin{align*}
A'B'C' &= \textsf{R}_OT_{P'}^{-1}(D_0E_0F_0)=\textsf{R}_OT_{P'}^{-1}K(ABC)\\
&=\textsf{R}_OT_{P'}^{-1}KT_{P}^{-1}(DEF)=\textsf{R}_O \textsf{M}^{-1}(DEF),
\end{align*}
where $\textsf{M}=T_P \circ K^{-1} \circ T_{P'}$ is a homothety or translation, by III, Theorem 3.4.  This shows that corresponding sides of $A'B'C'$ and $DEF$ are parallel, so these triangles are homothetic or congruent, by the converse of Desargues' Theorem.
\end{proof}

\begin{defn}  The {\bf tangential triangle} of $ABC$ with respect to the circumconic $\tilde{\C}_O$ is the triangle whose sides are tangent to $\tilde{\C}_O$ at the points $A, B, C$.
\end{defn}

\begin{thm}
\label{thm:Opers}
Assume that $P$ is an ordinary point, not lying on $\iota(l_\infty)$.  The point $O$ is the perspector of the circumcevian triangle of $Q$ and the tangential triangle of $ABC$ with respect to the conic $\tilde{\C}_O$.
\end{thm}

\begin{proof}
First assume $H$ is not a vertex, so that $O$ is distinct from the points $D_0, E_0, F_0$.  The point $O$ is the center of the conic $\tilde{\C}_O$, and $D_0$ is the midpoint of the chord $BC$ on this conic, so $OD_0$ lies on the pole $U$ of the line $BC$ and $b=UB$ and $c=UC$ are the tangents to $\tilde{\C}_O$ from $U$.  If $V$ and $W$ are the poles of $AC$ and $AB$, then $UVW$ is the tangential triangle of $ABC$, and $UVW$ is perspective to $D_0E_0F_0$ from the point $O$.  The assertion then follows from Proposition \ref{prop:circum}b).  If $H$ is a vertex, say $H=A$, then $O = D_0$ and $V$ and $W$ still lie on $OB'$ and $OC'$, respectively.  In this case $B$ and $C$ are antipodal points on $\tilde \cC_O$, so the tangents at those points are parallel and meet at $U \in l_\infty$.  To show that $UVW$ and $A'B'C'$ are perspective from $O$, we need to show that $OA'$ lies on $U$.  By Proposition \ref{prop:circum}c), $\tilde A=T_{P'}^{-1}(D_0)$ and $A'$ are antipodal points on $\tilde \cC_O$.  The tangents at $A'$ and $\tilde A$ are parallel to $BC$, so the pole of the line $OA'$ is the point at infinity on $BC$. Thus, $B, C$, and the pole of $OA'$ are collinear, which implies that $OA'$ and the tangents at (i.e. polars of) $B$ and $C$ are concurrent.
\end{proof}

\begin{lem}
If $\gamma$ denotes the isogonal map for triangle $ABC$, the map $\gamma \circ \gamma_P$ is a projective collineation fixing the vertices of $ABC$.
\end{lem}

\begin{proof}
Let $\rho_a, \rho_b, \rho_c$ denote the reflections in the angle bisectors of angles $A, B$, and $C$.  The definitions of the maps $\gamma$ and $\gamma_P$ imply that $\gamma \circ \gamma_P(R)$ is the intersection of the lines $A\rho_ah_a(R), B\rho_bh_b(R), C\rho_ch_c(R)$, for any point $R$ which is not on a side of $ABC$.  Now, if $R$ is a point on the side $AB$, for example, other than a vertex, this intersection is just $AB \cdot C\rho_ch_c(R)$.  Also, for the vertex $A$, for example, this intersection can naturally be considered to be $A$ itself, since the intersection $ B\rho_bh_b(A) \cdot C\rho_ch_c(A)=BA \cdot CA=A$, and $A\rho_ah_a(A)=A$.  Therefore, we can extend the definition of the map $\gamma \circ \gamma_P$ to coincide with the intersection of $A\rho_ah_a(R), B\rho_bh_b(R), C\rho_ch_c(R)$ for all points $R$.  \smallskip

Let the point $R$ vary on a line $l$, which does not lie on $A$ or $B$.  As in the proof of Lemma \ref{lem:mapl} we have the sequence of projectivities
\[\rho_ah_a(AR) \ \barwedge \ AR \ \stackrel{l}{\doublebarwedge} \ BR \ \barwedge \rho_bh_b(BR).\]
Hence, we have a projectivity $\phi: \rho_ah_a(AR) \ \barwedge \ \rho_bh_b(BR)$ between the pencil on $A$ and the pencil on $B$.  If $R$ is on the line $AB$, then $\rho_ah_a(AR)=\rho_ah_a(AB)=\rho_a(AC)=AB=\rho_bh_b(BA)$.  The common line $AB$ in the pencils on $A$ and $B$ is invariant, so $\phi$ is a perspectivity.  (See \cite{co2}, p. 35.)  Hence $\gamma \circ \gamma_P(R)=\rho_ah_a(AR) \cdot \rho_bh_b(BR)$ varies on a line.  If $l$ lies on $R=C$, then we have $\gamma \circ \gamma_P(C)=C$. \smallskip

This argument shows that $\gamma \circ \gamma_P(l)$ is a line whenever $l$ is not a side of $ABC$, and that $\gamma \circ \gamma_P$ fixes the vertices.  If $R$ varies on $AB$, then since $\rho_ah_a(AB)=\rho_bh_b(AB)=AB$, it is clear that $\gamma \circ \gamma_P(R) = AB \cdot \rho_ch_c(CR)$ varies on $AB$, so the sides are invariant lines.  Thus, $\gamma \circ \gamma_P$ is a collineation.  It is also easy to see this map is a projective collineation, for if $R$ varies on a line not through $A$, for example, then $\gamma \circ \gamma_P(R)$ varies on a line $m$, and
\[R \ \barwedge \ AR \ \barwedge \ \rho_ah_a(AR) \ \barwedge \ \gamma \circ \gamma_P(R)\]
is a projectivity between $l$ and $m$.  This proves the lemma.
\end{proof}

We are now ready to prove the main result of this paper.  Recall from \cite{mm3} that the generalized orthocenter $H$ for $P$ with respect to $ABC$ is the intersection of the lines through the vertices which are parallel, respectively, to the lines $QD, QE, QF$. \bigskip

\noindent {\bf The TCC-Perspector Theorem.} \smallskip
{\it Assume $P$ is an ordinary point and does not lie on $\iota(l_\infty)$.    If $H \neq A, B, C$ is the generalized orthocenter for $P$ with respect to triangle $ABC$, then the isogonal conjugate $\gamma(H)$ is the perspector of the tangential triangle of $ABC$ and the circumcevian triangle of $\gamma(Q)$, both taken with respect to the circumcircle of $ABC$.  In short, $\gamma(H)$ is the TCC-perspector of $\gamma(Q)$ with respect to $ABC$.}  \medskip

\begin{proof}
Apply the collineation $\gamma \circ \gamma_P$ to the result of Theorem \ref{thm:Opers}.  Then $\gamma \circ \gamma_P(O)$ is the perspector of the circumcevian triangle of $\gamma \circ \gamma_P(Q)=\gamma(Q)$ and the tangential triangle of $\gamma \circ \gamma_P(ABC)=ABC$ with respect to the conic $\gamma \circ \gamma_P(\tilde{\C}_O)$.  But Proposition \ref{prop:gammaCO} shows that $\gamma \circ \gamma_P(\tilde{\C}_O)=\gamma(\l_\infty)$ is just the circumcircle of $ABC$ (!), so $\gamma(H)$ is the TCC-perspector of $\gamma(Q)$, as claimed.
\end{proof}

\begin{figure}
\[\includegraphics[width=4.5in]{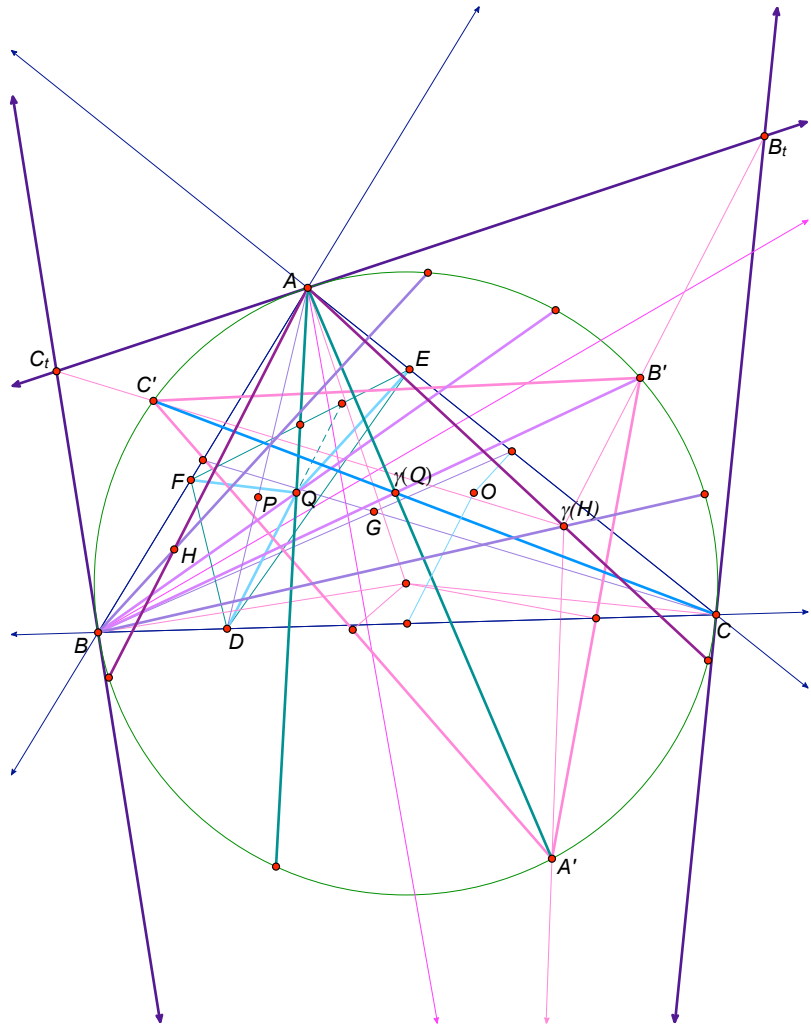}\]
\caption{TCC-perspector $\gamma(H)$ of circumcevian triangle $A'B'C'$ of $\gamma(Q)$ and tangential triangle $A_tB_tC_t$.  (Pink rays from $A$ and $B$ are angle bisectors.)}
\label{fig:4.1}
\end{figure}
\bigskip

This theorem yields the following formula for the TCC-perspector $T(Q)$ of a point $Q$:
\[T(Q)=\gamma \circ K^{-1} \circ T_{R}^{-1} \circ K \circ \gamma(Q),\]
where $R=K^{-1}(\gamma(Q))$.  This is reminiscent of the formula that was proved in Part I for the cyclocevian conjugate of a point $P$.  \medskip

In our final result we give an alternate characterization of the point $\gamma_P(G)$, which is a kind of supplement to the TCC-Perspector Theorem.

\begin{thm}
The perspector of $ABC$ and the tangential triangle $abc$ of $ABC$ with respect to the circumconic $\tilde{\C}_O$ is $\gamma_P(G)$.
\label{thm:persG}
\end{thm}
\begin{proof}
As in Theorem \ref{thm:Opers}, let $U = b \cdot c$, where $b$ and $c$ are the polars of $B$ and $C$ with respect to $\tilde \cC_O$.  It is enough to show that $\gamma_P(G)$ lies on $AU$; similar arguments for the other two sides imply the theorem. From the definition of $\gamma_P$, we just need to show that the lines $AU$ and $AG$ are harmonic conjugates with respect to the lines $AQ$ and $l_a=Q_bQ_c = T_{P'}^{-1}(BC)$.  Let $A'B'C'$ be the circumcevian triangle of $Q$ with respect to $ABC$ and $\tilde{\C}_O$, and let $\tilde A = T_{P'}^{-1}(D_0)$.  Parts b) and c) of Proposition \ref{prop:circum} show that the line $A'\tilde A$ lies on $D_0$ and $O$, hence also on $U$ ($OD_0$ and $UD_0$ both lie in the conjugate direction to $BC$ with respect to $\tilde{\C}_O$). Taking the section $\{U, D_0, A', \tilde A\}$ of the four lines $AU, AG, AQ, Q_bQ_c$ by the line $A'\tilde A$ (using the fact that $AQ \cap \tilde{\C}_O = \{A, A'\}$), it is enough to show that the points $U$ and $D_0$ are harmonic conjugates with respect to $A'$ and $\tilde A$. But $A'$ and $\tilde A$ are on $\tilde{\C}_O$ by Proposition \ref{prop:circum}c) and $D_0$ on $BC$ is conjugate to $U = b \cdot c$, so this follows from the involution of conjugate points on $A'\tilde A$.
\end{proof}

\end{section}

\begin{section}{Appendix: two lemmas.}

\begin{lem}
\label{lem:Kmid}
For any ordinary point $R$, the point $K(R)$ is the midpoint of $T_P(R)$ and $T_{P'}(R)$.
\end{lem}

\noindent {\bf Remark.} This lemma is easy to prove using barycentric coordinates.  See \cite{mo}.  It requires some ingenuity  to prove this synthetically.

\begin{proof}
We shall use the fact that if $\textsf{A}_1$ and $\textsf{A}_2$ are affine maps, then the mapping $\textsf{T}$ which takes a point $R$ to the midpoint of the segment joining $\textsf{A}_1(R)$ to $\textsf{A}_2(R)$ is affine.  If this is the case, then with $\textsf{A}_1=T_P$ and $\textsf{A}_2=T_{P'}$, we know that $\textsf{T}(A)=D_0, \textsf{T}(B)=E_0$, and $\textsf{T}(C)=F_0$.  Hence the affine mapping $\textsf{T}$ agrees with the affine mapping $K$ on three non-collinear points, so they must be the same map. \smallskip

Let  $\textsf{A}_1$ and $\textsf{A}_2$ be arbitrary affine maps, and for the purposes of this proof let $A,B,C$ be three collinear points, with $A_i=\textsf{A}_i(A), B_i=\textsf{A}_i(B), C_i=\textsf{A}_i(C)$.  We will prove that $\textsf{T}$ is affine.  We know that $A_i, B_i, C_i$ are collinear, for $i=1,2$, and the ratios $\frac{|A_1B_1|}{|B_1C_1|}=\frac{|A_2B_2|}{|B_2C_2|}$, since affine maps preserve ratios along a line.  The affine map $T=\textsf{A}_2 \circ \textsf{A}_1^{-1}$ induces a projectivity from $l_1=A_1B_1$ to $l_2=A_2B_2$.  We show first that the midpoints $I,J,K$ of $A_1A_2, B_1B_2$, and $C_1C_2$ are collinear (assuming $l_1 \neq l_2$; otherwise the assertion is trivial).  Let $F =l_1 \cdot l_2$.  The axis $m$ of this projectivity is the join of $T^{-1}(F)$ on $l_1$ and $T(F)$ on $l_2$. (See \cite{co2}, p. 37.)  Let $M_1$ and $M_2$ be the midpoints of the segments $FT^{-1}(F)$ and $T(F)F$, respectively, so that $T(M_1)=M_2$.  \smallskip

If $T(F)=F=T^{-1}(F)$ is an ordinary point, this is easy, because then the triangles $A_1FA_2, B_1FB_2$, and $C_1FC_2$ are all similar (by SAS), so $I, J, K$ are all collinear with $F$.  If $F$ lies on $l_\infty$, then it is also easy to see, because in that case the line $IJ$, which is parallel to $l_1$ and $l_2$, cuts $C_1C_2$ in the same ratio as it cuts $A_1A_2$ and $B_1B_2$, so $I, J, K$ are also collinear in this case.  (If $\textsf{T}$ is not injective on $AB$ and $I=J$, then it is easy to see that $A_1IB_1 \cong A_2IB_2$ by a half-turn, which implies easily that $\textsf{T}(AB)=I$, i.e., that $\textsf{T}$ maps the whole line $AB$ to the point $I$.)  \smallskip

Now assume $F$ is ordinary and $FT(F)T^{-1}(F)$ is a triangle.  Then $n=M_1M_2$ is parallel to the axis $m=T(F)T^{-1}(F)$.  Moreover, $I, J$ and $K$ are distinct.  We will show the points $I, J, K$ lie on the line $n$.  It is enough to prove this for the point $I$.  We redefine $I$ as the point $I=n \cdot A_1A_2$ and have to show $I$ is the midpoint of segment $A_1A_2$.  Let $S=M_1A_2 \cdot M_2A_1$ be the cross-join for the points $M_1, A_1$ and their images $M_2, A_2$ under the projectivity induced by $T$.  The point $S$ lies on the axis $m$.  We know that
$$\frac{|M_2F|}{|FA_2|} = \frac{|M_1T^{-1}(F)|}{|T^{-1}(F)A_1|} = \frac{|M_2S|}{|SA_1|}$$
by the fact that $T$ is affine and the similarity of the triangles $M_2A_1M_1$ and $SA_1T^{-1}(F)$.  Hence, $\frac{|M_2F|}{|FA_2|} = \frac{|M_2S|}{|SA_1|}$.  It follows that triangle $FM_2S$ is similar to $A_2M_2A_1$ and therefore $SF \pa A_1A_2$.  Then using the quadrangle $FM_1SM_2$, with diagonal points $A_1$ and $A_2$ and $I$ on $M_2M_1$ shows that $I$ is the harmonic conjugate of $A_1A_2 \cdot SF = A_1A_2 \cdot l_\infty$, so is the midpoint of $A_1A_2$.  Conversely, given the ordinary point $I \neq M_1, M_2, M_1'$ on the line $n=M_1M_2$, where $M_1'$ is the midpoint of segment $M_1M_2$, let $R$ be the harmonic conjugate of $I$ with respect to $M_1$ and $M_2$.  Then the ordinary point $R$ is also distinct from $M_1, M_2, M_1'$.  If $l$ is the line through $I$ parallel to $FR$, then $l$ intersects the line $l_1$ in a point $A_1$ and $l_2$ in a point $A_2$ such that
$$\frac{|M_1A_1|}{|M_1F|}=\frac{|M_1I|}{|M_1R|}=\frac{|M_2I|}{|M_2R|}=\frac{|M_2A_2|}{|M_2F|}=\frac{|M_2A_2|}{|M_2T(F)|},$$
from which we conclude that $T(A_1)=A_2$ and $I$ is the midpoint of $A_1A_2$.  Since $\textsf{T} \textsf{A}_1^{-1}$ maps $T^{-1}(F), M_1, F$ to $M_1, M_1', M_2$, respectively, we see that $\textsf{T}$ maps line $AB$ onto line $n$.  \smallskip

Thus, the mapping $\textsf{T}$ taking $R$ to the midpoint of $\textsf{A}_1(R)$ and $\textsf{A}_2(R)$ has the property that it maps every line either to a line or a point.  In the former case, if the map $T: l_1 \rightarrow l_2$ (with $l_1 \neq l_2$) has $F=l_1 \cdot l_2$ as a fixed point, then the mapping $\textsf{T}: AB \rightarrow IJ$ is projective.  If $F$ is ordinary, this holds because, as in the third paragraph of the proof above, the triangles $A_1FA_2, B_1FB_2$, and $C_1FC_2$ are similar, so $A_1A_2 \pa B_1B_2 \pa C_1C_2$.  If $O = A_1A_2 \cdot l_\infty$ is the point at infinity on these lines, then we have $A_1B_1C_1 \stackrel{O}{\doublebarwedge} IJK$ is projective, so $ABC \ \barwedge \ IJK$ is projective as well.  If $\textsf{T}$ is a bijection, then since $\textsf{T}$ transforms one range projectively, it is a projective collineation (see \cite{co2}, p. 50), and since it takes ordinary points to ordinary points, it fixes the line $l_\infty$ and must be an affine map.  If the point $F$ is infinite, then the lines $A_1A_2, B_1B_2$, and $C_1C_2$ are either parallel or concurrent, and the same conclusion follows.  \smallskip

In the case under consideration, namely $\textsf{A}_1=T_P$ and $\textsf{A}_2=T_{P'}$, it is not difficult to see that $\textsf{T}$ is a bijection on the whole plane.  Since $\textsf{T}$ maps the vertices $A, B, C$ (going back to our usual notation) to the distinct points $D_0,E_0,F_0$, the above argument shows that every point on the sides of the medial triangle $D_0E_0F_0$ is in the image of $\textsf{T}$.  This implies easily that the map $\textsf{T}$ is $1-1$.  If, for example, $\textsf{T}(U)=\textsf{T}(V)=L$ for ordinary points $U \neq V$, then $\textsf{T}$ maps the whole line $UV$ to the point $L$.  If $S$ is any point for which $\textsf{T}(S) = L' \neq L$, then $\textsf{T}$ maps the lines $US$ and $VS$ to the line $LL'$.  Since $UVS$ is a triangle, then every point of the plane maps to some point on $LL'$, which contradicts the fact that the points $D_0, E_0, F_0$ are in the image of $\textsf{T}$.  Thus, $\textsf{T}$ is $1-1$ and maps lines to lines.  Now the fact that $D_0E_0F_0$ is a triangle implies that every ordinary point is on a line joining two points in the image of $\textsf{T}$, so $\textsf{T}$ is surjective.  \smallskip

Furthermore, the map $T = T_{P'} T_P^{-1}=\lambda$ always has a fixed point.  If $P$ does not lie on a median of triangle $ABC$ or on the Steiner circumellipse $\iota(l_\infty)$, such a fixed point is the center $Z$ of the cevian conic $\C_P$, by II, Theorem 4.1.  If $P \in \iota(l_\infty)$, but does not lie on a median of $ABC$, the same conclusion holds by II, Theorem 4.3.  And if $P$ lies on a median, say on $AG$, where $G$ is the centroid, then the point $F=D_0=AG \cdot BC$ is a fixed point of the map $T=\lambda$.  Hence, $\textsf{T}$ induces a projectivity on every line.  This completes the proof of the lemma.
\end{proof}

The result we need this for is the following, which weakens the hypothesis of III, Proposition 3.12.

\begin{lem}
\label{lem:commute}
For any point $P$, not on the sides of triangles $ABC$ or $K^{-1}(ABC)$, the two maps $T_P \circ K^{-1}$ and $T_{P'} \circ K^{-1}$ commute with each other.  In particular, the map $\textsf{M}=T_P \circ K^{-1} \circ T_{P'}$ is symmetric in $P$ and $P'=\iota(P)$.
\end{lem}

\begin{proof}
The assertion is equivalent to $\textsf{M}=T_PK^{-1}T_{P'} = T_{P'}K^{-1}T_P$, or to
$$T_{P'}^{-1}T_PK^{-1} = K^{-1}T_PT_{P'}^{-1}.$$
By Lemma \ref{lem:Kmid}, for any ordinary point $R$, $R$ is the midpoint of $K^{-1}T_P(R)$ and $K^{-1}T_{P'}(R)$.  Replacing $R$ by $T_{P'}^{-1}(Y)$, we know that $T_{P'}^{-1}(Y)$ is the midpoint of $K^{-1}T_PT_{P'}^{-1}(Y)$ and $K^{-1}(Y)$.  On the other hand, $Y$ is the midpoint of $T_PK^{-1}(Y)$ and $T_{P'}K^{-1}(Y)$, so applying $T_{P'}^{-1}$ gives that  $T_{P'}^{-1}(Y)$ is the midpoint of $T_{P'}^{-1}T_PK^{-1}(Y)$ and $K^{-1}(Y)$.  It follows that $T_{P'}^{-1}T_PK^{-1}(Y)=K^{-1}T_PT_{P'}^{-1}(Y)$ for any ordinary point $Y$, and this implies the assertion.
\end{proof}

\end{section}

\noindent Dept. of Mathematics, Maloney Hall\\
Boston College\\
140 Commonwealth Ave., Chestnut Hill, Massachusetts, 02467-3806\\
{\it e-mail}: igor.minevich@bc.edu
\bigskip

\noindent Dept. of Mathematical Sciences\\
Indiana University - Purdue University at Indianapolis (IUPUI)\\
402 N. Blackford St., Indianapolis, Indiana, 46202\\
{\it e-mail}: pmorton@iupui.edu


\begin{thebibliography}{WWW}

\bibitem[1]{ac} N. Altshiller-Court, {\it College Geometry, An Introduction to the Modern Geometry of the Triangle and the Circle}, Barnes and Noble, New York, 1952. Reprint published by Dover.

\bibitem[2]{co1} H.S.M. Coxeter, {\it The Real Projective Plane}, McGraw-Hill Book Co., New York, 1949.

\bibitem[3]{co2} H.S.M. Coxeter, {\it Projective Geometry}, $2^{nd}$ edition, Springer, 1987.

\bibitem[4]{co3} H.S.M. Coxeter, {\it Introduction to Geometry}, John Wiley \& Sons, Inc., New York, 1969.

\bibitem[5]{dau} P.H. Daus, Isogonal and isotomic conjugates and their projective generalization, Amer. Math. Monthly 43 (1936), 160-164.

\bibitem[6]{dvl} K.R. Dean and F.M. van Lamoen, Geometric construction of reciprocal conjugations, Forum Geometricorum 1 (2001), 115-120.

\bibitem[7]{fg} M.D. Fox and J. R. Goggins, Cevian Axes and Related Curves, Math. Gazette 91 (2007), 2-26.

\bibitem[8]{gu} M. de Guzm\'an, An extension of the Wallace-Simson theorem: projecting in arbitrary dimensions, Amer. Math. Monthly 106 (1999), 574-580.

\bibitem[9]{ki1} C. Kimberling, Central points and central lines in the plane of a triangle, Mathematics Magazine 67 (1994), 163-187.

\bibitem[10]{ki2} C. Kimberling, {\it Encyclopedia of Triangle Centers}, at http://faculty.evansville.edu/ck6/encyclopedia/ETC.html.

\bibitem[11]{mm0} I. Minevich and P. Morton, Synthetic Cevian Geometry, preprint, IUPUI Math. Dept. Preprint Series pr09-01, 2009, \texttt{http://math.iupui.edu/research/research-preprints}.

\bibitem[12]{mm1} I. Minevich and P. Morton, Synthetic foundations of cevian geometry, I: Fixed points of affine maps in triangle geometry, \texttt{http://arXiv.org/abs/1504.00210}, to appear in the Journal of Geometry (2016), doi:10.1007/s00022-016-0324-4.

\bibitem[13]{mmq} I. Minevich and P. Morton, A quadrilateral half-turn theorem, Forum Geometricorum 16 (2016), 133-139.

\bibitem[14]{mm2} I. Minevich and P. Morton, Synthetic foundations of cevian geometry, II: The center of the cevian conic, \texttt{http://arXiv.org/abs/1505.05381}, International J. of Geometry 5 (2016), 22-38.

\bibitem[15]{mm3} I. Minevich and P. Morton, Synthetic foundations of cevian geometry, III: The generalized orthocenter, \texttt{http://arXiv.org/abs/1506.06253}, submitted, 2016.

\bibitem[16]{mmv} I. Minevich and P. Morton, Vertex positions of the generalized orthocenter and a related elliptic curve,  \texttt{http://arXiv.org/abs/1608.04614}.

\bibitem[17]{mo} P. Morton, Affine maps and Feuerbach's Theorem, IUPUI Math. Dept. Preprint Series  pr09-05, 2009, \texttt{http://math.iupui.edu/research/research-preprints}.

\end{thebibliography}
\end{document}